\documentclass[journal]{new-aiaa} 
\usepackage[utf8]{inputenc}

\usepackage{amsthm}
\usepackage{booktabs}
\usepackage{algorithm}
\usepackage{multirow}
\usepackage{xcolor}
\usepackage{algpseudocode}
\usepackage{graphicx}
\usepackage{subcaption}
\usepackage{cleveref}
\usepackage[version=4]{mhchem}
\usepackage{siunitx}
\usepackage{dsfont}

\newcommand{\mb}{\mathbb}
\newcommand{\mc}{\mathcal}

\usepackage{amsfonts}       
\usepackage{nicefrac}       
\usepackage{microtype}      
\usepackage{xcolor}         
\usepackage{framed}
\usepackage{comment}
\usepackage{xargs}                      
\usepackage{pgf, tikz}
\usepackage{pgfplots}
\usepackage{diagbox}
\usepackage{longtable,tabularx}
\newcommand{\changed}[1]{\textcolor{black}{#1}}

\newtheorem{definition}{Definition}

\newtheorem{corollary}{Corollary}
\newtheorem{proposition}{Proposition}
\newtheorem{assumption}{Assumption}
\newtheorem{remark}{Remark}

\newtheorem{theorem}{Theorem}
\title{Bayesian Ambiguity Contraction-based Adaptive Robust Markov Decision Processes for Adversarial Surveillance Missions\footnote{A preliminary version of this paper titled "Adaptive Robust Markov Decision Process for Wide-Area Surveillance with Collaborative Combat Aircraft" was presented at the 2026 AIAA SciTech Forum, Orlando, FL, January 12-16, 2026 \cite{previous_paper}.}}

\author{Jimin Choi\footnote{Ph.D. Candidate, Department of Aerospace Engineering\changed{; jiminch@umich.edu. AIAA Student Member (Corresponding Author).}} and Max Z. Li\footnote{Assistant Professor, Department of Aerospace Engineering, Department of Civil and Environmental Engineering, Department of Industrial and Operations Engineering, University of Michigan, Ann Arbor, Michigan 48109, AIAA Member.}}
\affil{University of Michigan, Ann Arbor, MI 48109, USA}

\begin{document}

\maketitle
\setcounter{footnote}{0}
\renewcommand{\thefootnote}{\arabic{footnote}}
\begin{abstract}
Collaborative Combat Aircraft (CCAs) are envisioned to enable autonomous Intelligence, Surveillance, and Reconnaissance (ISR) missions in contested environments, where adversaries may act strategically to deceive or evade detection. These missions pose challenges due to model uncertainty and the need for safe, real-time decision-making. Robust Markov Decision Processes (RMDPs) provide worst-case guarantees but are limited by static ambiguity sets that capture initial uncertainty without adapting to new observations. This paper presents an \emph{adaptive} RMDP framework tailored to ISR missions with CCAs. We introduce a mission-specific formulation in which the \changed{CCA alternates} between movement and sensing states. \changed{Adversarial uncertainty is represented by finite threat prototypes defining observation, exposure, and state-action transition models.} Our approach incrementally refines policies by eliminating inconsistent threat models, allowing \changed{the CCA to move from worst-case protection toward more efficient mission policies} while maintaining robustness. \changed{We establish fixed point stability, agreement with nominal planning as credible sets identify the true node-specific threat prototypes, and a conservative expected value guarantee while those prototypes remain credible.} Experiments under Gaussian and non-Gaussian threat models across diverse network topologies show higher mission rewards and fewer exposure events compared to nominal and static robust planners.
\end{abstract}

\section*{\changed{Nomenclature}}

\begin{longtable}{@{}p{0.28\textwidth}p{0.66\textwidth}@{}}
\changed{$G=(\mathcal{V},\mathcal{E})$} & \changed{Mission graph} \\
\changed{$v\in\mathcal{V}$} & \changed{Surveillance node representing a physical location} \\
\changed{$\mathcal{S}$} & \changed{State space of the CCA RMDP} \\
\changed{$\mathcal{A}$} & \changed{Action space of the CCA} \\
\changed{$\Delta(\mathcal{S})$} & \changed{Set of probability distributions over $\mathcal{S}$} \\
\changed{$\Theta$} & \changed{Finite set of candidate threat prototypes} \\
\changed{$\theta\in\Theta$} & \changed{Candidate threat prototype} \\
\changed{$b_{t,v}(\theta)$} & \changed{Posterior probability of threat prototype $\theta$ at node $v$ and time $t$} \\
\changed{$U_v^t$} & \changed{Credible set of threat prototypes at node $v$ and time $t$} \\
\changed{$U_{s,a}^t$} & \changed{Set of transition kernels induced by $U_v^t$ for state $s$ and action $a$} \\
\changed{$\mathcal{P}_t$} & \changed{Global rectangular ambiguity set at time $t$} \\
\changed{$p(o\mid a,\theta)$} & \changed{Observation distribution under action $a$ and threat prototype $\theta$} \\
\changed{$q(z\mid a,\theta)$} & \changed{Exposure score distribution under action $a$ and threat prototype $\theta$} \\
\changed{$P(\cdot\mid s,a;\theta)$} & \changed{Transition kernel induced by threat prototype $\theta$} \\
\changed{$u(o)$} & \changed{Sensing utility associated with observation $o$} \\
\changed{$\pi$} & \changed{CCA policy} \\
\changed{$V$} & \changed{Value function} \\
\changed{$\gamma$} & \changed{Discount factor} \\
\end{longtable}

\section{Introduction}

The Collaborative Combat Aircraft (CCA) concept has been widely recognized as a key enabler for future air combat and intelligence, surveillance, and reconnaissance (ISR) missions in contested and communication-degraded environments \cite{Gunzinger2024CCA,Penney2022CCA,CRS2024CCA}. By operating autonomously or semi-autonomously alongside crewed platforms, CCAs are expected to perform high-risk, distributed sensing missions that demand persistent situational awareness under adversarial conditions. ISR missions play a central role in maintaining tactical superiority in modern operations, particularly in distributed environments with limited communications and active electronic warfare \cite{Chizek2003,Smagh2020,AFDN25-1}.

However, autonomous ISR in contested environments poses fundamental decision-making challenges. Adversaries may employ mobile surface-to-air missile systems, intermittent radar activation, and deceptive emissions to obscure their true intent and location \cite{Adamy2006EW}. These behaviors violate the standard assumptions of known and stationary system dynamics, forcing ISR agents to operate under partial observability and epistemic uncertainty. Persistent surveillance, which lies at the core of ISR operations, further requires sustained coordination among heterogeneous sensing assets such as UAVs and UGVs \cite{4526242,drones6040094}. Although cooperative UAV–UGV frameworks, energy-aware routing, and persistent coverage strategies have been extensively studied \cite{persistent2019,9222146,9697374}, surveillance continuity remains highly vulnerable to environmental uncertainty and adversarial disruption \cite{Egorov_Kochenderfer_Uudmae_2016}.

From a decision-making perspective, ISR planning under adversarial uncertainty naturally takes the form of a sequential decision-making problem. Markov decision processes (MDPs) \changed{have} been widely adopted for collision avoidance, air traffic management, and autonomous flight safety applications \cite{doi:10.2514/6.2010-8040,7991332,doi:10.2514/1.G001822,9594329,8900719,doi:10.2514/1.G001743,tb_atm}. More recently, reinforcement learning (RL) has been explored for autonomous navigation, mission planning, and multi-agent coordination in aerospace systems \cite{MA2025103342,irl,doi:10.2514/6.2025-1933,10.1145/3301273}. While RL enables adaptive policy learning from data, its reliance on extensive exploration and the absence of formal safety guarantees limit its applicability to safety-critical ISR missions \cite{kirschner2020distributionally,Suilen2025,pmlr-v162-wang22at}.

Robust Markov decision processes (RMDPs) provide a principled alternative by explicitly accounting for uncertainty in the transition dynamics and optimizing policies for worst-case performance \cite{iyengar2005robust,ref:Nilim-05,wiesemann2013robust,le2007robust}. RMDPs have shown strong generalization properties in safety-critical domains where nominal models are unreliable \cite{wiesemann2013robust}. However, most existing RMDP methods operate in the offline setting with fixed ambiguity sets \cite{li2025towards,blanchet2023double}. Even in recent online extensions, either alternative notions of robust regret are adopted \cite{10.5555/3702676.3702729}, or restrictive structural assumptions such as loop-free MDPs are imposed \cite{sun2025online}. These assumptions exclude the cyclic dynamics that arise naturally in persistent ISR missions.

In practical ISR missions, agents continuously observe how threats sense, react, and reposition, leading to evolving uncertainty over adversarial behavior. This creates an inherent exploration–exploitation trade-off between refining the threat model and maintaining survivability. Recent works on uncertainty-aware autonomy and adaptive coordination in hazardous environments demonstrate the benefits of explicitly addressing this trade-off \cite{bilevel2025,choi2026integratedexplorationawareuavroute,uavugv2025bandit,airground2025}. However, existing approaches do not clearly capture this trade-off within a practical sequential decision-making framework for adversarial ISR settings.

Motivated by this gap, this paper proposes an adaptive RMDP framework tailored to CCA-enabled ISR missions. Building on the online robust planning formulation in \cite{sun2025online}, we integrate \changed{Bayesian belief updates and credible set contraction to refine uncertainty over local threat prototypes using online sensing observations}, while relaxing the loop-free structural assumption and introducing graph-based structure. This yields an adaptive robust planner that \changed{becomes progressively less conservative as the credible sets contract}, thereby balancing survivability and ISR effectiveness.

The main contributions of this work are threefold. First, we develop a mission-specific, graph-based RMDP formulation that captures spatial constraints, sensing limitations, and exposure risks inherent to CCA-enabled ISR operations. Second, we extend adaptive online robust planning to adversarial ISR missions through Bayesian \changed{credible set contraction and establish fixed point stability, a probabilistic lower bound on expected performance, asymptotic agreement with nominal planning, and computational complexity}. Third, we validate the proposed framework through \changed{three numerical experiments} that demonstrate improved ISR performance and reduced exposure relative to nominal and static robust baselines across varying threat models and graph structures.

\section{\changed{Mission Model and Mathematical Definitions}}

\subsection{\changed{Mission Scenario and Decision-Making Scope}}

Autonomous ISR missions in contested environments require aircraft to monitor large geographic regions that may contain uncertain adversarial threats. These threats can exhibit a wide range of adversarial behaviors, often acting strategically to expose or deceive ISR aircraft operating in the area. While the exact threat behavior is unknown, each ISR aircraft maintains a \changed{finite} set of known \changed{\emph{threat prototypes}} representing plausible \changed{threat behaviors}. \changed{Here, a threat prototype refers to a candidate probabilistic model of the observations, exposure outcomes, and state transitions associated with a threat. A formal definition is provided in Sec.~\ref{sec:threat_prototypes}.} \changed{We assume that prior intelligence identifies a finite number of recurring threat patterns, such as aggressive radar use or passive sensing, and that the corresponding threat prototypes are available to each aircraft before mission execution.}

\begin{figure}[htbp]
    \centering
    \includegraphics[width=\linewidth]{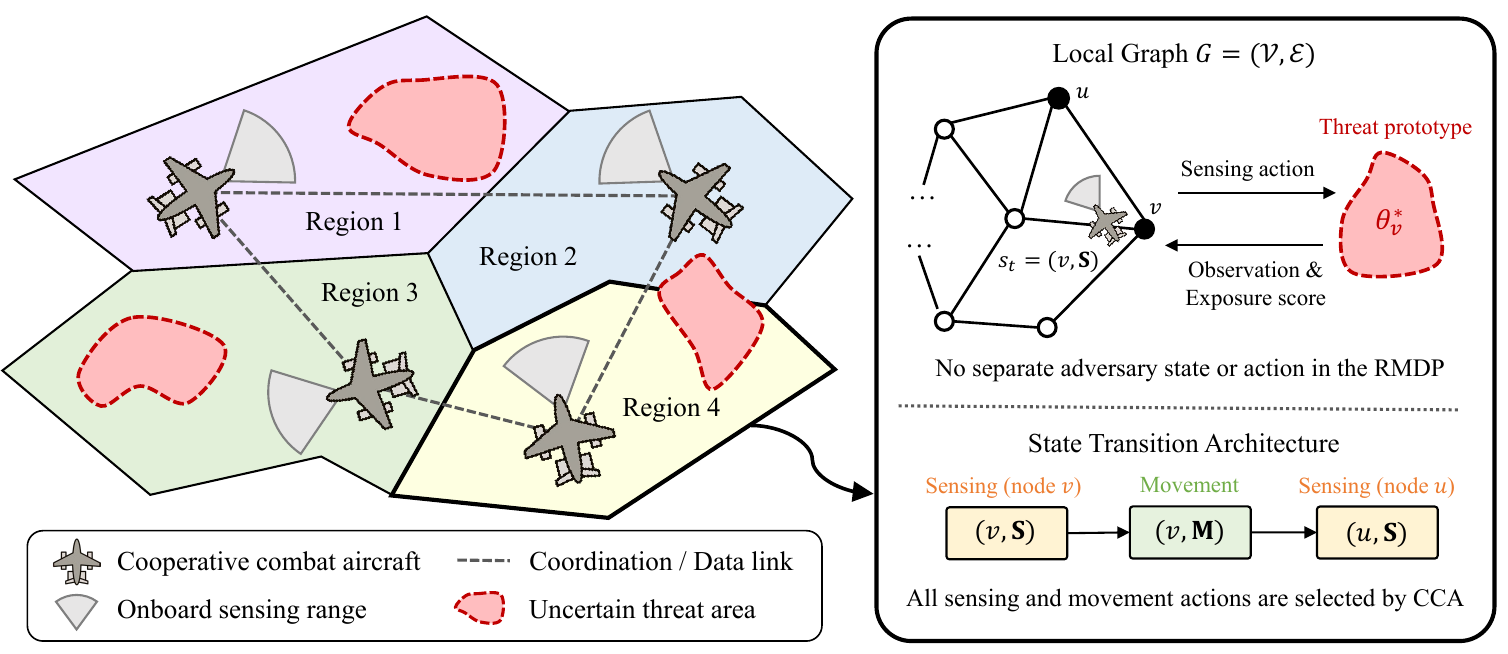}
    \caption{\changed{Mission concept and local RMDP representation for decentralized ISR operations with CCAs.}}    
    \label{fig:mission}
\end{figure}

In this setting, we consider a team of CCAs \changed{assigned to separate surveillance regions, with each CCA managing local uncertainty and exposure risk within its assigned region.} Each aircraft operates under model uncertainty and \changed{selects movement and sensing actions in real time} without access to the \changed{true threat prototype}. The key challenge is operating CCAs \changed{while limiting exposure risk and maximizing the information collected through distributed sensing}. The mission objective is to maximize the reliable collection of ISR information across the operational area while \changed{managing exposure risk under uncertain and structured threat behavior}.

\Cref{fig:mission} illustrates the mission-level concept: a surveillance task is carried out by a fleet of CCAs, each assigned to a portion of the operational space. Threats, represented by uncertain adversarial behavior, may exist anywhere within the region and follow one of several known behavioral prototypes. Each \changed{CCA} autonomously selects movement and sensing actions within its assigned region.

\paragraph{Scope of This Study.}
While the mission concept in \Cref{fig:mission} involves a team of CCAs, this work does not directly model \changed{coordination among multiple agents} or task allocation across regions. Instead, we focus on the decision-making problem of a single autonomous \changed{CCA} operating within its assigned region, assuming that the overall area has been divided among aircraft before mission execution. Each \changed{CCA} independently performs its ISR task under local threat uncertainty and applies the proposed adaptive robust framework within its assigned region.

\subsection{\changed{Environment Graph, States, and Actions}}
ISR missions involve sequential decision-making under uncertainty, where the \changed{CCA} balances information gathering with exposure risk. To model this setting formally, we represent the operational environment as a graph-based RMDP. \changed{We consider an undirected graph $G=(\mc V,\mc E)$, where $\mc V=\{0,\ldots,V-1\}$ is the set of surveillance locations and $\mc E$ is the set of feasible movements between locations.} Each node corresponds to a spatial location with an unknown \changed{threat prototype}, and transitions occur through movement along the graph edges. \changed{The threat prototypes and their probabilistic models are formally defined in Sec.~\ref{sec:threat_prototypes}.} The aircraft takes sensing actions when located at a node and movement actions along the graph edges to navigate between adjacent surveillance locations.

\changed{For each node $v\in\mc V$, the neighboring nodes are defined as $\mc N(v)=\{u\in\mc V:(v,u)\in\mc E\}$. A graph node represents a physical surveillance location, whereas an RMDP state specifies both the current node and the current operational phase.} \changed{We therefore divide each decision cycle into two operational phases: a sensing phase, in which the CCA collects information at its current node, and a movement phase, in which it selects an adjacent node and relocates before sensing again. Accordingly,} we define the two-phase state space
\begin{equation}
\changed{\mc S=\{(v,\mathrm{S}),(v,\mathrm{M}):v\in\mc V\}.}
\end{equation}
\changed{Here, $(v,\mathrm{S})$ denotes the sensing state at node $v$, and $(v,\mathrm{M})$ denotes the movement state at node $v$.} At \changed{sensing states $(v,\mathrm{S})$}, the action set is \changed{$\mc A_{\mathrm{S}}=\{a_1,\ldots,a_{n_{\mathrm S}}\}$, where each $a_i$ denotes an available sensing mode or sensing direction}. At \changed{movement states $(v,\mathrm{M})$}, the action set is
\begin{equation}
\changed{\mc A_{\mathrm{M}}(v)=
\begin{cases}
\mc N(v), & \mc N(v)\neq\emptyset,\\
\{v\}, & \mc N(v)=\emptyset.
\end{cases}}
\end{equation}
\changed{The state-dependent action set is therefore defined as}
\begin{equation}
\changed{\mc A(s)=
\begin{cases}
\mc A_{\mathrm{S}}, & s=(v,\mathrm{S}),\\
\mc A_{\mathrm{M}}(v), & s=(v,\mathrm{M}).
\end{cases}}
\end{equation}
\changed{The overall CCA action space and the set of feasible state and action pairs are defined as}
\begin{equation}
\changed{\mc A=\bigcup_{s\in\mc S}\mc A(s), \qquad \mc K=\{(s,a):s\in\mc S,\ a\in\mc A(s)\}.}
\end{equation}
\changed{The initial state $s_0\in\mc S$ specifies the initial surveillance node and operational phase of the CCA.} \changed{All sensing and movement actions are selected by the CCA. The threat is not modeled as a second decision-making agent with a separate action space.} This two-phase design reflects real ISR missions, in which remaining in one place for extended periods increases the risk of exposure, defined as the event that the aircraft becomes observable to adversarial sensors. Alternating between sensing and movement prevents this buildup of risk and reduces prolonged exposure within a single region.

\subsection{\changed{Threat Prototypes and Probabilistic Models}}
\label{sec:threat_prototypes}

\begin{definition}[\changed{Threat Prototype}]
\changed{Let $\Theta=\{\theta_1,\ldots,\theta_{n_\theta}\}$ denote the finite set of candidate threat prototypes. Each prototype $\theta\in\Theta$ specifies an observation model, an exposure model, and a transition kernel associated with a recurring pattern of threat behavior, such as continuous radar emission, intermittent activation, or passive monitoring. Each node $v\in\mc V$ is associated with an unknown true threat prototype $\theta_v^\star\in\Theta$.}
\end{definition}

When taking a \changed{sensing} action $a_t\in\mc A_{\mathrm S}$ at node $v$, the \changed{CCA} conducts local reconnaissance and receives a scalar observation $o_{v,t}$ that quantifies the information gained from that area.
\begin{equation}
\changed{o_{v,t}\sim p(o\mid a_t,\theta_v^\star),}
\end{equation}
where \changed{$p(o\mid a,\theta)$ denotes the observation distribution under sensing action $a$ and threat prototype $\theta$}. In this context, exposure refers to the ISR aircraft being detected or otherwise compromised by the local threat. \changed{An exposure score is generated according to}
\begin{equation}
\changed{z_{v,t}\sim q(z\mid a_t,\theta_v^\star),}
\end{equation}
where \changed{$q(z\mid a,\theta)$ represents the exposure score distribution under sensing action $a$ and threat prototype $\theta$}. This score induces a binary exposure event
\begin{equation}
\eta_{v,t}=\mathds{1}[z_{v,t}>\tau_{\eta}],
\end{equation}
for a fixed threshold $\tau_{\eta}>0$. The corresponding exposure probability is given by
\begin{equation}
p_{\mathrm{exp}}(a,\theta)=\int_{\tau_{\eta}}^{\infty}q(z\mid a,\theta)\,dz,
\end{equation}
which represents the probability that the exposure score exceeds the threshold $\tau_{\eta}$. \changed{The observation and exposure models are specified using prior intelligence available before mission execution. We assume that the observation and exposure outcomes are conditionally independent given the sensing action and threat prototype.}

\changed{Let $\Delta(\mc S)$ denote the set of probability distributions over the state space $\mc S$. A transition kernel assigns a next state distribution to every feasible state and action pair and is therefore an element of}
\begin{equation}
\changed{\prod_{(s,a)\in\mc K}\Delta(\mc S).}
\end{equation}
\changed{Each threat prototype induces a local next state distribution according to}
\begin{equation}
\changed{P(\cdot\mid s,a;\theta)=\phi_{s,a}(\theta)\in\Delta(\mc S).}
\end{equation}
\changed{Here, $\phi_{s,a}$ maps a threat prototype to the next state distribution associated with state $s$ and action $a$.}

\subsection{\changed{Beliefs, Credible Sets, and Ambiguity Sets}}
\label{sec:belief_sets}

\changed{For each node $v\in\mc V$, the CCA maintains a posterior belief over the candidate threat prototypes, defined as}
\begin{equation}
\changed{b_{t,v}(\theta)=\Pr\!\left(\theta_v^\star=\theta\mid\mc D_{t,v}\right), \qquad \theta\in\Theta,}
\end{equation}
\changed{where $\mc D_{t,v}$ denotes the sensing data collected at node $v$ through time $t$.}

\changed{To represent and update this uncertainty,} we propose an online adaptive RMDP framework in which uncertainty over \changed{threat behavior} is represented through node-wise credible sets $U_v^t$. Each credible set is initialized as $U_v^0=\Theta$, which contains all \changed{threat prototypes} based on prior knowledge. These credible sets induce the corresponding transition ambiguity sets $U_{s,a}^t$ through a fixed mapping $U_{s,a}^t=\phi_{s,a}(U_v^t)$, so that at the beginning of the mission the ambiguity sets cover the entire range of transition models consistent with the prototype set. \changed{The Bayesian update and the rule used to remove statistically implausible prototypes from $U_v^t$ are presented in Sec.~\ref{sec:belief_update}.} \changed{The CCA} therefore begins with a conservative strategy that assumes the worst-case threat among known prototypes. Over time, local observations enable \changed{the CCA} to eliminate incompatible \changed{threat prototypes}, gradually refining its belief and adapting its behavior. As the credible sets shrink, the induced ambiguity sets contract accordingly, enabling a transition from a cautious policy to more confident and efficient ISR operations.

The credible set $U_v^t$ therefore captures the set of \changed{threat prototypes} that remain plausible for node $v$ under the current posterior. For any state and action pair $(s,a)$ associated with node $v$, these credible sets are mapped to local transition ambiguity sets through the continuous map
\begin{equation}
\changed{U_{s,a}^t=\phi_{s,a}(U_v^t)=\left\{P(\cdot\mid s,a;\theta):\theta\in U_v^t\right\}.}
\end{equation}
\changed{Thus, $U_v^t$ contains threat prototypes, whereas $U_{s,a}^t$ contains the transition kernels induced by those prototypes.}

\begin{definition}[\changed{$(s,a)$ Rectangular Ambiguity Set}]
\changed{A transition ambiguity set $\mc P_t\subseteq\prod_{(s,a)\in\mc K}\Delta(\mc S)$ is $(s,a)$ rectangular if it can be expressed as}
\begin{equation}
\changed{\mc P_t=\prod_{(s,a)\in\mc K}U_{s,a}^t,}
\end{equation}
\changed{where $U_{s,a}^t\subseteq\Delta(\mc S)$ is the local transition ambiguity set associated with the feasible state and action pair $(s,a)\in\mc K$.}
\end{definition}

\changed{Here, $\prod$ denotes the Cartesian product over all feasible state and action pairs. Following the standard RMDP literature \cite{ref:Nilim-05,iyengar2005robust}, this structure permits an independent selection of one transition kernel from $U_{s,a}^t$ for each feasible pair $(s,a)\in\mc K$. In this work, $\mc P_t$ is the rectangular hull of the local transition ambiguity sets induced by the node-wise credible sets.} \changed{The key distinctions between the credible sets and transition ambiguity sets are summarized in Table~\ref{tab:credible_ambiguity_sets}.}

\setcounter{table}{0}
\begin{table}[htbp]
\centering
\caption{Credible sets and ambiguity sets in the adaptive RMDP.}
\label{tab:credible_ambiguity_sets}
\begin{tabular}{p{1.7cm} p{6.7cm} p{6.6cm}}
\toprule
 & \textbf{Credible Set $U_v^t$} & \textbf{\changed{Transition} Ambiguity Set $U_{s,a}^t$} \\
\midrule
\textit{Definition} & \changed{Threat prototypes that remain plausible at node $v$} & Transition models induced by $U_v^t$ via $\phi_{s,a}$ \\
\textit{Uncertainty} & Model uncertainty from limited observations & Transition uncertainty captured in the RMDP \\
\textit{Update} & Shrinks from posterior updates and pruning & Contracts through $U_{s,a}^t=\phi_{s,a}(U_v^t)$ \\
\textit{Usage} & Determines remaining prototype hypotheses & Used in the robust Bellman operator \\
\bottomrule
\end{tabular}
\end{table}

\section{\changed{Adaptive Robust Planning Method}}


\subsection{\changed{Reward Model}}

We design the reward to capture four operational factors. Sensing utility measures how informative an observation is at a location, reflecting the surveillance value gained from the action. Exploration novelty encourages visiting \changed{less visited} nodes to expand situational awareness rather than repeatedly sampling familiar areas. Exposure risk penalizes actions that increase the chance of detection or compromise, accounting for both immediate and persistent effects of past exposure. Finally, operational cost reflects resource burdens such as fuel usage or reduced endurance. \changed{We index the reward by time because the novelty and exposure history variables evolve during mission execution. During each planning update, the current values of $n_t(v)$, $e_t$, and $d_t$ are treated as fixed mission statistics.}

\changed{For a sensing action $a\in\mc A_{\mathrm S}$ at node $v$, the expected sensing reward under threat prototype $\theta$ is defined as}
\begin{equation}
\changed{
r_{\mathrm{sense}}^t(v,a,\theta)
=
r_{\mathrm{obs}}(v,a,\theta)
+
r_{\mathrm{nov}}^t(v)
+
r_{\mathrm{exp}}^t(a,\theta)
-
c_{\mathrm{sense}}(a).
}
\end{equation}
\changed{The observation reward, novelty reward, and exposure penalty are defined as}
\begin{equation}
\begin{aligned}
r_{\mathrm{obs}}(v,a,\theta)&=
\mathbb E_{p(o\mid a,\theta)}[u(o)],\\
r_{\mathrm{nov}}^t(v)&=
\lambda_{\mathrm{nov}}n_t(v),\\
r_{\mathrm{exp}}^t(a,\theta)&=
-\lambda_{\mathrm{imm}}p_{\mathrm{exp}}(a,\theta)
-\lambda_{\mathrm{pers}}e_t
-\lambda_{\mathrm{cum}}\log(1+d_t),\\
c_{\mathrm{sense}}(a)&\ge 0,
\end{aligned}
\end{equation}
\changed{Here, $u:\mathbb R\to\mathbb R$ is a sensing utility function. We use $u(o)=o$ in the numerical experiments.} \changed{The observation reward $r_{\mathrm{obs}}(v,a,\theta)$ quantifies the expected sensing utility under threat prototype $\theta$. The immediate exposure penalty is evaluated using $p_{\mathrm{exp}}(a,\theta)$ because the planner evaluates the expected reward before the exposure outcome is observed. After the sensing action is executed, the realized exposure indicator $\eta_{v,t}\in\{0,1]\}$ updates the persistent exposure level and cumulative exposure count according to}
\begin{equation}
\changed{
e_{t+1}=\max\{\alpha e_t,\eta_{v,t}\},
\qquad
d_{t+1}=d_t+\eta_{v,t},
}
\end{equation}
\changed{where $e_t\in[0,1]$, $d_t\ge0$, and $\alpha\in(0,1)$ controls the decay of the persistent exposure level.}

The novelty map $n_t(v)\ge0$ increases when a node has not been visited recently. We implement it as an exponential moving average of the time since the last visit,
\begin{equation}
\label{novelty-update}
n_{t+1}(v)=\beta n_t(v)+(1-\beta)\bigl(t-t_{\mathrm{last}}(v)\bigr),
\end{equation}
where $t_{\mathrm{last}}(v)$ denotes the most recent visit time and $\beta\in(0,1)$ controls the decay rate. An exponential moving average avoids abrupt jumps that occur when a node is visited once after a long absence. This smooth update provides a simple \changed{and} stable measure of \changed{the elapsed time since the node was last visited}.

\changed{For a movement action $u\in\mc A_{\mathrm M}(v)$ at node $v$, the movement reward is defined as}
\begin{equation}
\changed{
r_{\mathrm{move}}^t(v,u)
=
-c_{\mathrm{move}}
+
\lambda_{\mathrm{nov}}n_t(u)
-
\lambda_{\mathrm{pers}}e_t
-
\lambda_{\mathrm{cum}}\log(1+d_t),
}
\end{equation}
\changed{where $c_{\mathrm{move}}\ge0$ is the movement cost. The movement reward encourages travel toward less visited nodes while accounting for movement cost and the persistent and cumulative effects of previous exposure events. No immediate exposure term is included because a new exposure outcome is generated only when the CCA performs a sensing action.}

\changed{The weights $\lambda_{\mathrm{imm}}$, $\lambda_{\mathrm{pers}}$, $\lambda_{\mathrm{cum}}$, and $\lambda_{\mathrm{nov}}$ are nonnegative and determine the relative importance of immediate exposure, persistent exposure, cumulative exposure, and exploration novelty, respectively. At the beginning of the mission, we initialize $e_0=0$, $d_0=0$, and $n_0(v)=0$ for all $v\in\mc V$.}

\subsection{\changed{Robust and Nominal Bellman Operators}}
\label{sec:bellman_operators}

\changed{Let $\Pi$ denote the set of stationary policies, where each policy $\pi\in\Pi$ assigns an action distribution over $\mc A(s)$ to every state $s\in\mc S$.} \changed{For a fixed planning time $t$ and transition kernel $P$, the value of a stationary policy $\pi$ is defined as}
\begin{equation}
\changed{V_{\pi,t}^{P}(s)=\mb E_\pi^P\!\left[\sum_{k=0}^{\infty}\gamma^k r^t(S_k,A_k)\mid S_0=s\right],}
\end{equation}
where \changed{$\gamma\in(0,1)$ is the discount factor}.

\changed{At each planning time $t$, the credible sets $U_v^t$, novelty map $n_t(v)$, persistent exposure level $e_t$, and cumulative exposure count $d_t$ are held fixed while the CCA computes a policy. For a state $s$ associated with node $v$, the stage reward is defined as}
\begin{equation}
\changed{
r^t(s,a;\theta)=
\begin{cases}
r_{\mathrm{sense}}^t(v,a,\theta), & s=(v,\mathrm S),\\
r_{\mathrm{move}}^t(v,a), & s=(v,\mathrm M).
\end{cases}}
\label{eq:stage_reward}
\end{equation}
\changed{The movement reward does not explicitly depend on the threat prototype, although its associated transition kernel may depend on $\theta$ through $P(\cdot\mid s,a;\theta)$.}

\changed{For a sensing action $a\in\mc A_{\mathrm S}$ at node $v$, the robust immediate reward surrogate is defined as}
\begin{equation}
\changed{\widetilde r_{\mathrm{rob}}^t(v,a,U_v^t)=\min_{\theta\in U_v^t}r_{\mathrm{sense}}^t(v,a,\theta).}
\label{eq:robust_reward_surrogate}
\end{equation}
\changed{The maximum posterior threat prototype at node $v$ is defined as}
\begin{equation}
\changed{\widehat\theta_{t,v}\in\arg\max_{\theta\in U_v^t}b_{t,v}(\theta),}
\label{eq:map_prototype}
\end{equation}
\changed{and the corresponding nominal immediate reward is}
\begin{equation}
\changed{\widetilde r_{\mathrm{nom}}^t(v,a)=r_{\mathrm{sense}}^t(v,a,\widehat\theta_{t,v}).}
\label{eq:nominal_reward_surrogate}
\end{equation}
\changed{Thus, the robust surrogate evaluates the least favorable immediate sensing reward over $U_v^t$, whereas the nominal counterpart evaluates the immediate sensing reward under a single maximum posterior threat prototype. These quantities are introduced to distinguish the robust and nominal immediate reward evaluations. The robust Bellman operator does not insert $\widetilde r_{\mathrm{rob}}^t$ separately, but jointly minimizes the stage reward and continuation value over the same threat prototype.}

Let $V:\mc S\to\mathbb R$ denote the value function over the \changed{two-phase} state space. During each Bellman update, auxiliary quantities that depend on visit history or past exposure, such as the novelty map $n_t(v)$, the persistent exposure level $e_t$, and the cumulative \changed{exposure count} $d_t$, are considered fixed at the current time step.

\changed{For a state $s$ associated with node $v$, the robust Bellman operator at planning time $t$ is defined as}
\begin{equation}
\changed{
(\mc T_{\mathrm r}^tV)(s)
=
\max_{a\in\mc A(s)}
\min_{\theta\in U_v^t}
\left\{
r^t(s,a;\theta)
+
\gamma\sum_{s'\in\mc S}P(s'\mid s,a;\theta)V(s')
\right\}.
}
\label{eq:robust_bellman_operator}
\end{equation}
\changed{Within each local minimization, the same threat prototype $\theta$ determines both the stage reward and the transition kernel. Therefore, the robust backup does not combine a reward induced by one prototype with a transition kernel induced by another prototype. For each feasible state and action pair, the inner minimization selects one prototype from the credible set associated with its node. Consistent with the rectangular hull $\mc P_t$, different feasible state and action pairs may attain their local minima under different threat prototypes.}

\changed{The nominal Bellman operator evaluates the reward and transition kernel under the maximum posterior threat prototype and is defined as}
\begin{equation}
\changed{
(\mc T_{\mathrm n}^tV)(s)
=
\max_{a\in\mc A(s)}
\left\{
r^t(s,a;\widehat\theta_{t,v})
+
\gamma\sum_{s'\in\mc S}P(s'\mid s,a;\widehat\theta_{t,v})V(s')
\right\}.
}
\label{eq:nominal_bellman_operator}
\end{equation}
Intuitively, \changed{$\mc T_{\mathrm r}^t$ performs a worst-case backup over all threat prototypes retained in $U_v^t$, while $\mc T_{\mathrm n}^t$ performs a standard Bellman backup under the maximum posterior threat prototype}. \changed{When $U_v^t=\{\widehat\theta_{t,v}\}$, the robust and nominal operators coincide for states associated with node $v$.}

Repeated application of the Bellman update,
\begin{equation}
V_{k+1}=\mc T_{\xi}^tV_k,\qquad \xi\in\{\mathrm r,\mathrm n\},
\label{eq:robust_nominal_value_iteration}
\end{equation}
converges to a unique fixed point \changed{when the credible sets and mission statistics are fixed at planning time $t$}. The resulting value function defines a greedy policy $\pi_{\xi}^\star$ that is stationary and optimal with respect to the chosen operator. \changed{The fixed point stability and robustness properties of these operators are established in Sec.~\ref{sec:theoretical_analysis}.}

\subsection{\changed{Bayesian Belief Update and Credible Set Contraction}}
\label{sec:belief_update}

At each sensing step at node $v$, the \changed{CCA} maintains a belief distribution $b_{t,v}$ over the candidate \changed{threat prototypes}. After taking sensing action $a_t$ and observing $o_{v,t}$, the belief is updated according to
\begin{equation}
b_{t+1,v}(\theta)=\frac{b_{t,v}(\theta)p(o_{v,t}\mid a_t,\theta)}{\sum_{\theta'\in\Theta}b_{t,v}(\theta')p(o_{v,t}\mid a_t,\theta')}, \qquad \theta\in\Theta,
\label{eq:bayesian_belief_update}
\end{equation}
where $p(o_{v,t}\mid a_t,\theta)$ denotes the observation likelihood under sensing action $a_t$ and threat prototype $\theta$. \changed{Beliefs at nodes other than the currently sensed node remain unchanged.}

\changed{Let the maximum posterior threat prototype be}
\begin{equation}
\changed{\widehat\theta_{t+1,v}\in\operatorname*{arg\,max}_{\theta\in U_v^t}b_{t+1,v}(\theta).}
\label{eq:posterior_map}
\end{equation}
The credible set is then refined according to the posterior confidence rule
\begin{equation}
U_v^{t+1}=
\begin{cases}
\changed{\{\widehat\theta_{t+1,v}\}}, & \text{if } b_{t+1,v}(\widehat\theta_{t+1,v})\ge\rho_{\mathrm{lock}},\\
\changed{\{\theta\in U_v^t:b_{t+1,v}(\theta)\ge\varepsilon_{\mathrm{prune}}\}\cup\{\widehat\theta_{t+1,v}\}}, & \text{otherwise},
\end{cases}
\label{eq:credible_set_update}
\end{equation}
where $\rho_{\mathrm{lock}}\in(0,1)$ is the posterior confidence threshold used to lock the credible set to a single prototype, and $\varepsilon_{\mathrm{prune}}\in(0,1)$ is the posterior threshold used to remove implausible prototypes. \changed{The maximum posterior prototype is retained in the second case to ensure that the updated credible set is nonempty.}

The rule commits to the most probable threat prototype once its posterior probability exceeds the locking threshold, while removing prototypes with sufficiently low posterior probabilities otherwise. \changed{Because the update only removes elements from the current credible set,}
\begin{equation}
\changed{U_v^{t+1}\subseteq U_v^t.}
\end{equation}
\changed{For every feasible state and action pair $(s,a)$ associated with node $v$, the updated credible set induces}
\begin{equation}
\changed{U_{s,a}^{t+1}=\phi_{s,a}(U_v^{t+1})=\left\{P(\cdot\mid s,a;\theta):\theta\in U_v^{t+1}\right\}.}
\label{eq:ambiguity_set_update}
\end{equation}
\changed{Consequently, $U_{s,a}^{t+1}\subseteq U_{s,a}^t$, so contraction of the credible set produces contraction of the corresponding transition ambiguity sets. For nodes that are not sensed at time $t$, the beliefs, credible sets, and induced ambiguity sets remain unchanged.}

Routine surveillance observations provide the data used to distinguish plausible threat prototypes from implausible ones. As the credible sets contract, the robust planner considers fewer threat prototypes and can become less conservative while retaining robustness with respect to the prototypes that remain in the credible sets. The $\textsc{Shrink}$ operator in \Cref{alg:adaptive_robust} implements the update rule in \eqref{eq:credible_set_update}.

\subsection{\changed{Adaptive Planning Algorithm}}
\label{sec:adaptive_algorithm}

\changed{\Cref{fig:pipeline} summarizes the information flow between sensing, inference, ambiguity contraction, and robust planning. A sensing observation updates the posterior belief and credible set at the current node. The updated credible set induces local transition ambiguity sets, which are used by the robust Bellman operator during periodic replanning.}

\begin{figure}[htbp]
\centering
\begin{tikzpicture}[
    node distance=4.0cm,
    box/.style={rectangle, rounded corners, minimum width=2.8cm, minimum height=1cm, draw=black!80, thick, align=center}
]
\node[box] (obs) {Observation\\ \changed{$o_{v,t}$}};
\node[box, right of=obs] (cred) {\changed{Belief and Credible Set}\\ \changed{$b_{t,v},\,U_v^t$}};
\node[box, right of=cred] (amb) {\changed{Transition Ambiguity Set}\\ $U_{s,a}^t$};
\node[box, right of=amb] (bell) {Robust Bellman\\ Operator};

\draw[->, thick] (obs) -- (cred);
\draw[->, thick] (cred) -- (amb);
\draw[->, thick] (amb) -- (bell);
\end{tikzpicture}
\caption{\changed{Information flow in the adaptive RMDP framework. Sensing observations update the local belief and credible set, which induce the transition ambiguity sets used during robust replanning.}}
\label{fig:pipeline}
\end{figure}

The adaptive robust planner replans every $K$ steps based on the current visit statistics and ambiguity sets, as outlined in \Cref{alg:adaptive_robust}. This loop implements a \changed{closed loop} interaction between inference and planning: Bayesian updates refine \changed{local beliefs}, posterior \changed{updates contract the credible sets and their induced transition ambiguity sets}, and periodic robust replanning adapts the policy to the current uncertainty. \changed{During early mission stages, the resulting policy remains robust with respect to all threat prototypes retained in the credible sets.} \changed{The parameter $T$ denotes the number of mission execution steps, and $K$ denotes the replanning interval.}

\begin{algorithm}[h!]
\caption{Adaptive Robust Planning for ISR}
\label{alg:adaptive_robust}
\begin{algorithmic}[1]
\State Initialize uniform beliefs $b_{0,v}(\theta)=1/|\Theta|$ and credible sets $U_v^0=\Theta$ for all \changed{$v\in\mc V$}
\State \changed{Initialize $U_{s,a}^0=\phi_{s,a}(U_v^0)$ for every feasible pair $(s,a)\in\mc K$ associated with node $v$}
\State Initialize exposure variables $e_0=0$, $d_0=0$, and novelty map $n_0(v)=0$ for all \changed{$v\in\mc V$}
\State \changed{Initialize the mission state $s_0\in\mc S$}
\For{$t=0,1,\ldots,T-1$}
  \If{$t\equiv0\pmod K$}
    \State \changed{Construct $\mc P_t=\prod_{(s,a)\in\mc K}U_{s,a}^t$}
    \State \changed{Apply robust value iteration using $\mc T_{\mathrm r}^t$ to obtain policy $\pi_t$}
  \EndIf
  \State \changed{Observe the current state $s_t=(v_t,\sigma_t)$, where $\sigma_t\in\{\mathrm S,\mathrm M\}$}
  \State \changed{Select action $a_t=\pi_t(s_t)\in\mc A(s_t)$}
  \If{$\sigma_t=\mathrm S$}
    \State \changed{Sample $o_{v_t,t}\sim p(o\mid a_t,\theta_{v_t}^\star)$ and $z_{v_t,t}\sim q(z\mid a_t,\theta_{v_t}^\star)$}
    \State \changed{Set $\eta_{v_t,t}=\mathds{1}[z_{v_t,t}>\tau_\eta]$}
    \State Update belief:
    \[
    {
    b_{t+1,v_t}(\theta)
    =
    \frac{b_{t,v_t}(\theta)p(o_{v_t,t}\mid a_t,\theta)}
    {\sum_{\theta'\in\Theta}b_{t,v_t}(\theta')p(o_{v_t,t}\mid a_t,\theta')}
    }
    \]
    \State \changed{Set $b_{t+1,v}(\theta)=b_{t,v}(\theta)$ for all $v\neq v_t$}
    \State Update credible set:
    \[
    U_{v_t}^{t+1}\gets\textsc{Shrink}(U_{v_t}^t,b_{t+1,v_t}),
    \qquad
    U_v^{t+1}\gets U_v^t
    \text{ for }v\neq v_t
    \]
    \State Update induced ambiguity sets:
    \[
    \changed{
    U_{s,a}^{t+1}
    \gets
    \phi_{s,a}(U_{v_t}^{t+1})
    \quad
    \text{for all $(s,a)\in\mc K$ associated with $v_t$
    }
    }
    \]
    \State \changed{Keep all other local ambiguity sets unchanged}
  \Else
    \State \changed{Set $\eta_{v_t,t}=0$, $b_{t+1,v}=b_{t,v}$, and $U_v^{t+1}=U_v^t$ for all $v\in\mc V$}
    \State \changed{Keep $U_{s,a}^{t+1}=U_{s,a}^t$ for all $(s,a)\in\mc K$}
  \EndIf
  \State Update exposure variables:
  \[
  e_{t+1}=\max\{\alpha e_t,\eta_{v_t,t}\},
  \qquad
  d_{t+1}=d_t+\eta_{v_t,t}
  \]
  \State \changed{Sample the next state $s_{t+1}\sim P(\cdot\mid s_t,a_t;\theta_{v_t}^\star)$}
  \State Update novelty map $n_{t+1}(v)$ based on visit recency according to~\eqref{novelty-update}
\EndFor
\end{algorithmic}
\end{algorithm}

As uncertainty diminishes and local ambiguity sets contract, the adaptive robust planner considers fewer threat prototypes and becomes less conservative. \changed{When each credible set contracts to a singleton containing the maximum posterior prototype, the robust and nominal Bellman operators coincide. Under the assumptions stated in Sec.~\ref{sec:theoretical_analysis}, the fixed point stability and robust value results show that the adaptive planner approaches the corresponding nominal solution as ambiguity contracts, while retaining a conservative value guarantee with respect to the prototypes that remain in the credible sets before contraction is complete.}

\section{\changed{Theoretical Analysis}}
\label{sec:theoretical_analysis}

We establish guarantees for the adaptive robust planner. Throughout, we work with finite state and action sets. The \changed{two-phase state space $\mc S$ is also finite.}

\subsection{Assumptions and Regularity Conditions}
\label{sec:theoretical_assumptions}

\begin{assumption}[Discounting and bounded rewards]\label{ass:disc}
The discount factor satisfies $\gamma\in(0,1)$. Per step rewards and costs are uniformly bounded and measurable. \changed{At each fixed planning time $t$, the mission statistics $n_t(v)$, $e_t$, and $d_t$ are finite and are held fixed during value iteration. The resulting stage reward satisfies}
\begin{equation}
\changed{\max_{(s,a)\in\mc K}\max_{\theta\in\Theta}|r^t(s,a;\theta)|\le R_t<\infty.}
\end{equation}
\end{assumption}

\begin{assumption}[\changed{$(s,a)$ Rectangular Uncertainty and Contraction}]\label{ass:rect}
For any $(s,a)$ rectangular family of ambiguity sets \changed{$\mc P_t=\prod_{(s,a)\in\mc K}U_{s,a}^t$}, where \changed{$U_{s,a}^t\subseteq\Delta(\mc S)$ denotes the local transition ambiguity set associated with the feasible pair $(s,a)\in\mc K$}, the robust Bellman operator
\begin{equation}
\changed{
(\mc T_{\mathrm r}^tV)(s)
=
\max_{a\in\mc A(s)}
\min_{\theta\in U_v^t}
\left\{
r^t(s,a;\theta)
+
\gamma\sum_{s'\in\mc S}P(s'\mid s,a;\theta)V(s')
\right\}
}
\end{equation}
is a $\gamma$ contraction on \changed{$(\mathbb R^{|\mc S|},\|\cdot\|_\infty)$} with modulus $\gamma$, \changed{where state $s$ is associated with node $v$. Within each local minimization, the same threat prototype $\theta$ determines both the stage reward and the transition kernel.}
\end{assumption}

\begin{assumption}[Posterior and credible sets]\label{ass:post}
\changed{For every node $v\in\mc V$, the Bayesian posterior belief $b_{t,v}(\theta)=\Pr(\theta_v^\star=\theta\mid\mc D_{t,v})$ is well-defined over the finite prototype set $\Theta$, where $\mc D_{t,v}$ denotes the sensing data collected at node $v$ through time $t$. Each credible set $U_v^t$ is a nonempty subset of $\Theta$. Because $\Theta$ is finite, every credible set is compact and Borel measurable.}
\end{assumption}

\begin{assumption}[\changed{Posterior Consistency and Eventual Nesting}]\label{ass:tight}
For each $v$, there exists a \changed{true threat prototype} $\theta_v^\star\in\Theta$ such that
\begin{equation}
\changed{\Pr\!\left(\theta_v^\star\in U_v^t\right)\to1}
\end{equation}
and
\begin{equation}
\changed{\Pr\!\left(U_v^t=\{\theta_v^\star\}\right)\to1}
\end{equation}
\changed{as the number of sensing observations collected at node $v$ tends to infinity. That is, the probabilities of retaining the true prototype and identifying it uniquely both converge to one.} There exists an \changed{almost surely finite random time} after which $U_v^{t+1}\subseteq U_v^t$ holds for all subsequent updates at node $v$.
\end{assumption}

\begin{assumption}[Stepwise continuity and dominated convergence]\label{ass:cont}
For any \changed{feasible pair $(s,a)\in\mc K$}, the maps \changed{$\theta\mapsto r^t(s,a;\theta)$ and $\theta\mapsto P(s'\mid s,a;\theta)$ for every $s'\in\mc S$} are continuous on $\Theta$. \changed{Here, $\mc O$ denotes the observation space and $\mu$ denotes a reference measure on $\mc O$.} \changed{In particular, for a sensing state, the maps $\theta\mapsto\mathbb E_{p(o\mid a,\theta)}[u(o)]$ and $\theta\mapsto p_{\mathrm{exp}}(a,\theta)$ are continuous.} If the observation space is infinite, there exists an integrable dominating function $g(o)$ such that
\begin{equation}
\changed{|u(o)|p(o\mid a,\theta)\le g(o),\qquad \int_{\mc O}g(o)\,\mu(\mathrm do)<\infty}
\end{equation}
\changed{for every $\theta\in\Theta$, so that dominated convergence applies. Hence, for compact $U\subseteq\Theta$ and bounded $V$, the map}
\begin{equation}
\changed{
U\mapsto
\min_{\theta\in U}
\left\{
r^t(s,a;\theta)
+
\gamma\sum_{s'\in\mc S}P(s'\mid s,a;\theta)V(s')
\right\}
}
\end{equation}
\changed{is continuous under set convergence in the Painlev\'e--Kuratowski sense.}\footnote{\changed{Formally, Painlev\'e--Kuratowski convergence is defined through convergence of the inner and outer limits of sets \cite{AubinFrankowska1990}. In the present finite prototype setting, $U_v^t\to\{\theta_v^\star\}$ means that the credible set eventually retains only the true threat prototype.}}
\end{assumption}

\begin{assumption}[\changed{Greedy Selection and Tie Breaking}]\label{ass:greedy}
Action sets $\mc A(s)$ are finite for all \changed{$s\in\mc S$}. A measurable \changed{tie-breaking} rule is fixed. If the optimal action is unique for each state under the true \changed{threat prototype $\theta_v^\star$ associated with the node of that state}, then the greedy selector will pick a single, \changed{well-defined} action in the limit.
\end{assumption}

These assumptions jointly ensure that the robust Bellman operator is well-defined, contractive, and stable under learning. Discounting turns the infinite horizon objective into a convergent power series and bounds the influence of remote outcomes, while uniform boundedness ensures that the Bellman update remains within a compact set. Bounded one stage rewards imply that any fixed point \changed{at planning time $t$} satisfies
\begin{equation}
\changed{\|V\|_\infty\le\frac{R_t}{1-\gamma},}
\end{equation}
\changed{where $R_t:=\max_{(s,a)\in\mc K}\max_{\theta\in\Theta}|r^t(s,a;\theta)|$. This uniform bound ensures that $V$ remains in the closed ball of $(\mathbb R^{|\mc S|},\|\cdot\|_\infty)$} and allows Banach's fixed point theorem\footnote{\changed{Banach's fixed point theorem states the following:} Let $(X,d)$ be a complete metric space and let $T:X\to X$ be a contraction. There exists $\gamma\in(0,1)$ such that $d(Tx,Ty)\le\gamma d(x,y)$ for all $x,y\in X$. Then $T$ admits a unique fixed point $x^\star\in X$, and the iterates $T^k x_0$ converge to $x^\star$ for any initial point $x_0\in X$ \cite{10.5555/528623}.} to apply without additional growth conditions. If these assumptions are violated, for example, when $\gamma\ge1$ or the rewards are unbounded, value iteration may diverge or fail to admit a finite fixed point. Even if a policy is mathematically well-defined, small modeling errors can accumulate and undermine \changed{robustness} guarantees. In most engineered systems, physical limits or budget caps impose natural bounds. We ensure that any additional penalty terms remain bounded \changed{at each planning time} so that the bound on the value function remains valid.

Rectangularity decouples ambiguity across \changed{state and action pairs}, which removes adversarial coupling effects across time. This decoupling is exactly what makes the inner minimization separable and ensures the Bellman recursion holds in the presence of uncertainty over the transition probabilities. \changed{In the proposed robust Bellman operator, each local minimization evaluates the reward and transition kernel induced by the same threat prototype.} Contraction with modulus $\gamma$ ensures uniqueness of the robust value and geometric error decay for value iteration and policy evaluation. Without contraction, limit points may not exist, and set valued Bellman maps can exhibit cycling, highlighting the role of boundedness and discounting. Rectangularity is a simplifying but conservative assumption: it ignores correlations between transitions of different $(s,a)$ pairs, making the problem more tractable but potentially less tight.

\changed{The finite prototype set and nonempty credible sets} ensure that Bayesian learning can be embedded in the planning process. \changed{Since $\Theta$ is finite, each credible set $U_v^t\subseteq\Theta$ is automatically compact and Borel measurable, and every minimum over $U_v^t$ is attained.} Measurability ensures that the data-driven sets can be embedded into the filtration, a family of time indexed $\sigma$ algebras encoding the information revealed up to time $t$, so that the operators are random but well-defined. \changed{The finite state, action, and prototype sets also ensure the existence of greedy actions and local minimizing prototypes.}

Posterior-consistency captures the idea that uncertainty shrinks as learning progresses. When each credible set only shrinks over time, the corresponding Bellman operators stop oscillating and approach a stable limit. Together, they imply that the algorithm eventually uses the Bellman operator corresponding to the true \changed{threat prototype}. \changed{In the finite prototype setting, consistency means that the probability that $U_v^t$ equals the singleton $\{\theta_v^\star\}$ approaches one.} This is what underpins operator convergence in Corollary~\ref{cor:asym_opt}. Eventual nesting allows monotonicity arguments that simplify the proof of \changed{convergence of the fixed points}. If credible sets expand intermittently due to model misspecification, the derived \changed{robustness} bounds no longer \changed{provide} useful protection. If sets shrink to the wrong \changed{prototype} with positive probability, the limit policy can be biased.

Continuity ensures that small changes in \changed{the retained threat models produce only small changes in the corresponding reward and transition evaluations}, preventing discontinuous control switches as uncertainty decreases. Dominated convergence is the technical bridge that lets limits pass through expectations, which is essential because the Bellman operator computes \changed{rewards and transition expectations} that depend on the \changed{threat prototype}. With domination,\footnote{By dominated convergence, pointwise convergence of a sequence of integrands together with an integrable dominating envelope allows the limit to pass through the expectation under the standard conditions of dominated convergence \cite{Billingsley1999}.} \changed{limits can pass through the expected sensing utility}, which is used to show continuity of the inner minima with respect to set limits. This is what yields uniform convergence of operators on bounded sets. In our setting, the \changed{prototype set} $\Theta$ is discrete, so \changed{continuity with respect to $\theta$ holds automatically when $\Theta$ is equipped with the discrete topology}. \changed{The domination condition is needed only when the observation space is infinite.} We retain this assumption for generality, so that the same theoretical statements extend without modification to continuous parameter spaces or models with continuous observations.

Finite action sets and a fixed tie-breaking rule prevent policy instability. The algorithm repeatedly selects greedy actions, \changed{and multiple actions may attain the same maximum value.} A fixed rule avoids random policy fluctuations that would confound convergence and policy limit arguments. Finiteness makes the argmax operator upper hemicontinuous\footnote{A set valued map $F:X\rightrightarrows Y$ assigns to each $x\in X$ a subset $F(x)\subseteq Y$. It is upper hemicontinuous at $x$ if for every open set $O$ containing $F(x)$, there exists a neighborhood $N$ of $x$ such that $F(x')\subseteq O$ for all $x'\in N$. Intuitively, it means that the graph of $F$ does not expand discontinuously in the limit, so the possible argmax sets vary continuously under small perturbations \cite{AubinFrankowska1990}.} with nonempty, compact values. A measurable selector then exists, meaning that a Borel measurable state-to-action mapping can be selected from each argmax set. Uniqueness in the limit upgrades set convergence of \changed{state and action values} to pointwise convergence of deterministic policies.

\subsection{\changed{Fixed Point Stability under Shrinking Ambiguity Sets}}
\label{sec:operator_convergence}

As the ambiguity sets evolve with data, the robust Bellman operator itself changes over time. A key question is whether the robust solution approaches the corresponding nominal solution as uncertainty contracts. \changed{At each planning time $t$, the credible sets and mission statistics are held fixed while value iteration is performed. We therefore analyze the fixed point of the robust operator $\mc T_{\mathrm r}^t$ and compare it with the fixed point of the nominal operator $\mc T_{\mathrm n}^t$ defined using the same mission statistics. This comparison does not require the novelty and exposure statistics to converge over mission time.}

\changed{For a state $s\in\mc S$, let $v(s)$ denote its associated node. Define the maximum reward and transition discrepancies within the current credible sets as}
\begin{equation}
\changed{
\Delta_{r,t}
=
\max_{(s,a)\in\mc K}
\max_{\theta\in U_{v(s)}^t}
\left|
r^t(s,a;\theta)
-
r^t(s,a;\widehat\theta_{t,v(s)})
\right|
}
\label{eq:reward_discrepancy}
\end{equation}
\changed{and}
\begin{equation}
\changed{
\Delta_{P,t}
=
\max_{(s,a)\in\mc K}
\max_{\theta\in U_{v(s)}^t}
\left\|
P(\cdot\mid s,a;\theta)
-
P(\cdot\mid s,a;\widehat\theta_{t,v(s)})
\right\|_1.
}
\label{eq:transition_discrepancy}
\end{equation}

\begin{theorem}[\changed{Fixed Point Stability under Ambiguity Contraction}]
\label{thm:as_convergence}
\changed{Suppose Assumptions~\ref{ass:disc}--\ref{ass:cont} hold. At each fixed planning time $t$, let $V_{\mathrm r}^{t,\star}$ and $V_{\mathrm n}^{t,\star}$ denote the unique fixed points of $\mc T_{\mathrm r}^t$ and $\mc T_{\mathrm n}^t$, respectively. Let}
\begin{equation}
\changed{
M_t=\frac{R_t}{1-\gamma},
}
\end{equation}
\changed{where $R_t$ is the stage reward bound in Assumption~\ref{ass:disc}. Then}
\begin{equation}
\changed{
\left\|
V_{\mathrm r}^{t,\star}
-
V_{\mathrm n}^{t,\star}
\right\|_\infty
\le
\frac{
\Delta_{r,t}
+
\gamma M_t\Delta_{P,t}
}{
1-\gamma
}.
}
\label{eq:fixed_point_stability_bound}
\end{equation}
\changed{Consequently, whenever $\Delta_{r,t}\to0$ and $M_t\Delta_{P,t}\to0$, the robust and nominal fixed points satisfy}
\begin{equation}
\changed{
\left\|
V_{\mathrm r}^{t,\star}
-
V_{\mathrm n}^{t,\star}
\right\|_\infty
\to0.
}
\end{equation}
\end{theorem}

\begin{proof}
See Appendix~\ref{sec:proof}.
\end{proof}

\begin{remark}[\changed{Interpretation of Theorem~\ref{thm:as_convergence}}]
\changed{Theorem~\ref{thm:as_convergence} compares the robust and nominal planning problems at the same planning time. Both operators therefore use the same novelty map, persistent exposure level, and cumulative exposure count. The difference between their fixed points is caused only by the remaining uncertainty over threat prototypes. As the credible sets contract, the range of reward and transition models considered by the robust operator decreases, which reduces $\Delta_{r,t}$ and $\Delta_{P,t}$.}
\end{remark}

\changed{Because the prototype set $\Theta$ is finite and the credible sets are nested, convergence of $U_v^t$ to the singleton $\{\theta_v^\star\}$ implies that the credible set eventually equals that singleton. At that point, the robust and nominal Bellman operators coincide exactly. The following corollary formalizes this result.}

\begin{corollary}[\changed{Asymptotic Agreement with Nominal Planning}]
\label{cor:asym_opt}
Under Assumptions~\ref{ass:disc}--\ref{ass:greedy}, suppose that
\begin{equation}
\changed{
U_v^t\longrightarrow\{\theta_v^\star\}
}
\end{equation}
almost surely for every node $v\in\mc V$. \changed{Because $\Theta$ is finite and the credible sets are nested and nonempty, there exists, almost surely, a finite time after which}
\begin{equation}
\changed{
U_v^t=\{\theta_v^\star\}
\qquad
\text{for every }v\in\mc V.
}
\end{equation}
\changed{For every sufficiently large $t$, the robust Bellman operator therefore equals the nominal operator under the true node-specific threat prototypes,}
\begin{equation}
\changed{
(\mc T_{\mathrm n}^{t,\theta^\star}V)(s)
=
\max_{a\in\mc A(s)}
\left\{
r^t(s,a;\theta_{v(s)}^\star)
+
\gamma
\sum_{s'\in\mc S}
P(s'\mid s,a;\theta_{v(s)}^\star)V(s')
\right\}.
}
\label{eq:true_nominal_operator}
\end{equation}
\changed{Hence, almost surely,}
\begin{equation}
\changed{
\mc T_{\mathrm r}^t
=
\mc T_{\mathrm n}^{t,\theta^\star}
}
\end{equation}
\changed{and}
\begin{equation}
\changed{
V_{\mathrm r}^{t,\star}
=
V_{\mathrm n}^{t,\theta^\star,\star}
}
\end{equation}
\changed{for every sufficiently large planning time $t$. In particular,}
\begin{equation}
\changed{
\left\|
V_{\mathrm r}^{t,\star}
-
V_{\mathrm n}^{t,\theta^\star,\star}
\right\|_\infty
\to0
}
\end{equation}
\changed{almost surely. If the maximizing action of the true nominal operator is unique at every state, the corresponding robust and nominal greedy policies also coincide for every sufficiently large planning time.}
\end{corollary}

\begin{remark}[\changed{Operational Meaning of Asymptotic Agreement}]
As the credible sets contract to the true \changed{threat prototypes}, uncertainty in \changed{both the reward and transition evaluations} disappears, and the robust formulation \changed{becomes identical to the nominal planning problem defined by the true prototype at each node}. The algorithm therefore transitions from conservative \changed{decision-making} to nominal \changed{decision-making} as threat uncertainty is resolved. \changed{This result compares the robust and nominal solutions at each planning time and does not require the time varying mission statistics $n_t(v)$, $e_t$, and $d_t$ to converge to fixed limits. A single limiting value function or policy can be asserted only if those mission statistics also stabilize.}
\end{remark}

\subsection{Safety and Computational Properties}
\label{sec:safety_complexity}

While \changed{asymptotic agreement with nominal planning} characterizes the long run behavior of the adaptive planner, safety concerns arise during learning: the model is uncertain, and policies are selected based on ambiguity sets rather than the true \changed{threat prototypes}. A natural question is whether the robust value used for planning provides conservative guarantees on \changed{expected performance} under the true dynamics. The next proposition shows that Assumption~\ref{ass:rect} yields such a probabilistic \changed{robust value guarantee}.

\begin{proposition}[\changed{Probabilistic Robust Value Guarantee under $(s,a)$ Rectangular Uncertainty}]
\label{prop:safety}
\changed{For each state $s\in\mc S$, let $v(s)$ denote its associated node. At planning time $t$, define the local joint model set}
\begin{equation}
\changed{
\mc M_{s,a}^t
=
\left\{
\left(
r^t(s,a;\theta),
P(\cdot\mid s,a;\theta)
\right)
:
\theta\in U_{v(s)}^t
\right\},
\qquad
(s,a)\in\mc K,
}
\label{eq:local_joint_model_set}
\end{equation}
\changed{and define the corresponding rectangular model set as}
\begin{equation}
\changed{
\mc M_t
=
\prod_{(s,a)\in\mc K}\mc M_{s,a}^t.
}
\label{eq:global_joint_model_set}
\end{equation}
\changed{For a stationary policy $\pi$ and a model $M\in\mc M_t$, let $V_{\pi,t}^{M}(s)$ denote the expected discounted value under the local reward and transition pairs selected by $M$. Define the robust policy value as}
\begin{equation}
\changed{
\underline V_{\pi,t}^{\mathrm r}(s)
=
\min_{M\in\mc M_t}V_{\pi,t}^{M}(s).
}
\label{eq:robust_policy_value}
\end{equation}
\changed{Let the credible event be}
\begin{equation}
\changed{
\mc E_t
=
\bigcap_{v\in\mc V}
\left\{
\theta_v^\star\in U_v^t
\right\},
\qquad
\delta_t
=
\Pr(\mc E_t^c).
}
\label{eq:credible_event}
\end{equation}
\changed{Let $V_{\pi,t}^{\theta^\star}(s)$ denote the expected discounted value under the true node-specific threat prototypes, where the reward and transition kernel at state $s$ are given by $r^t(s,a;\theta_{v(s)}^\star)$ and $P(\cdot\mid s,a;\theta_{v(s)}^\star)$, respectively. Then, for any stationary policy $\pi_t$, including a policy greedy with respect to the robust Bellman operator,}
\begin{equation}
\changed{
\Pr\left(
\forall s\in\mc S:
V_{\pi_t,t}^{\theta^\star}(s)
\ge
\underline V_{\pi_t,t}^{\mathrm r}(s)
\right)
\ge
1-\delta_t.
}
\label{eq:probabilistic_robust_value_guarantee}
\end{equation}
\end{proposition}

\begin{proof}
See Appendix~\ref{sec:proof}.
\end{proof}

\begin{remark}[\changed{Robust Value Guarantee}]
Conditioned on the event that the true \changed{threat prototypes} lie in the credible sets, the \changed{expected discounted value under the true reward and transition model} dominates the worst-case value used for planning. In other words, the planner's predicted robust value is a conservative lower bound on its \changed{expected performance under the true model}. The shortfall term $\delta_t$ is the probability that \changed{at least one credible set does not include its true threat prototype} at time $t$. Improving statistical coverage is the primary lever to tighten the guarantee. Under $(s,a)$ rectangularity, the inner minimization is exact for each local \changed{reward and transition pair}, so the robust planner evaluates the worst-case retained prototype without approximation. The result holds uniformly with respect to the state and for \changed{any stationary policy}, so it applies to both policy evaluation and improvement.
\end{remark}

Beyond safety, it is important to understand the computational cost of implementing the adaptive robust planner. Each replan operates with a fixed family of credible sets and performs robust value iteration until convergence. The following proposition provides a per replan complexity bound and clarifies how the uncertainty size affects the computational cost.

\begin{proposition}[Per Replan Complexity]
\label{prop:complexity}
During each replan, the family of credible sets $\{U_v^t\}$ is fixed, and robust value iteration is performed until the value updates are within a tolerance $\varepsilon$. Let
\begin{equation}
\changed{
\overline{|U^t|}
=
\frac{1}{|\mc K|}
\sum_{(s,a)\in\mc K}
|U_{v(s)}^t|
}
\end{equation}
\changed{denote the average credible set size over feasible state and action pairs, and let}
\begin{equation}
\changed{
\overline d_t
=
\frac{
\displaystyle
\sum_{(s,a)\in\mc K}
\sum_{\theta\in U_{v(s)}^t}
\left|\operatorname{supp}P(\cdot\mid s,a;\theta)\right|
}{
\displaystyle
\sum_{(s,a)\in\mc K}
|U_{v(s)}^t|
}
}
\end{equation}
\changed{where $\operatorname{supp}P(\cdot\mid s,a;\theta)=\{s'\in\mc S:P(s'\mid s,a;\theta)>0\}$ is the support of the transition distribution. The equation is the average transition support size over the prototype models evaluated during a Bellman sweep. Assuming that prototype-dependent reward quantities are precomputed, the number of value iteration steps required to reach tolerance $\varepsilon$ satisfies}
\begin{equation}
\changed{
N_\varepsilon
=
\mathcal O\left(
\frac{\log(C_t/\varepsilon)}{1-\gamma}
\right),
}
\end{equation}
\changed{where $C_t$ depends on the initial value error and the reward bound $R_t$. The resulting computational cost of one replan is}
\begin{equation}
\changed{
\mathcal O\left(
N_\varepsilon
|\mc K|
\overline{|U^t|}
\overline d_t
\right).
}
\label{eq:per_replan_complexity}
\end{equation}
\changed{For dense transition kernels, $\overline d_t=\mathcal O(|\mc S|)$, and the worst-case cost becomes}
\begin{equation}
\changed{
\mathcal O\left(
N_\varepsilon
|\mc K|
\overline{|U^t|}
|\mc S|
\right).
}
\end{equation}
\end{proposition}

\begin{proof}
See Appendix~\ref{sec:proof}.
\end{proof}

\begin{remark}[Interpretation of Complexity]
The \changed{factors in \eqref{eq:per_replan_complexity}} reflect distinct computations within each replan: \changed{$|\mc K|$ accounts for the feasible state and action pairs, $\overline{|U^t|}$ accounts for the retained threat prototypes evaluated in each local minimization, and $\overline d_t$ accounts for the transition expectation under each prototype.} Together, \changed{these factors} represent the total per replan workload. As learning progresses and uncertainty about the \changed{threat prototypes} decreases, the algorithm needs to evaluate fewer models within each \changed{credible set}. Consequently, \changed{as the credible sets contract toward singletons}, the computational cost per replan approaches that of standard value iteration under a single nominal model.
\end{remark}

\section{\changed{Numerical Experiments}}\label{sec:experiments}
We conduct three numerical experiments under different environmental and structural conditions to validate the proposed adaptive robust planner. Experiment~1 tests whether the planner operates as intended and converges in a controlled environment with Gaussian threat models. Experiment~2 tests robustness and generalization to non-Gaussian threat models, including Gaussian mixtures and log-normal distributions. Experiment~3 evaluates the scalability and consistency of the planner as the network topology and size vary. \changed{In all numerical experiments, the phase and movement transitions are independent of the threat prototype. Thus, $P(\cdot\mid s,a;\theta)=P(\cdot\mid s,a)$ for all $\theta\in\Theta$, and the experiments represent a special case of the general formulation in which threat uncertainty affects the sensing reward through the observation and exposure models. The theoretical formulation also permits prototype-dependent transition kernels.}

\subsection{Baseline Implementations}
To evaluate the performance of the proposed adaptive robust planner, we compare it against two baselines: \textit{static robust} and \textit{nominal} planning. These baselines are two complementary extremes in terms of uncertainty handling and adaptability. Both baselines share the same \changed{state and action structure and mission statistic updates} as Algorithm~\ref{alg:adaptive_robust}, but differ in how they treat and propagate uncertainty during planning.

In the static robust baseline, the ambiguity sets of possible threat \changed{prototypes} remain fixed throughout the entire mission. The logic of the static robust planner is summarized in \Cref{alg:static_robust}. It optimizes actions through a minmax value-iteration process that selects the best decision under the worst-case threat configuration within each static set. As a result, it provides \changed{a conservative reference policy under fixed credible sets} but is often overly conservative. Since the uncertainty representation is never refined, the planner tends to overestimate risk and produce cautious but suboptimal behavior. 

\begin{algorithm}[htbp]
\caption{Static Robust Planning}
\label{alg:static_robust}
\begin{algorithmic}[1]
\State \changed{Initialize fixed credible sets $\widebar U_v=\Theta$ for all $v\in\mc V$}
\State \changed{Initialize $U_{s,a}=\phi_{s,a}(\widebar U_v)$ for every feasible pair $(s,a)\in\mc K$ associated with node $v$}
\For{$t=0,1,\ldots,T-1$}
    \If{$t\equiv0\pmod K$}
        \State \changed{Apply robust value iteration using the fixed credible sets $\{\widebar U_v\}_{v\in\mc V}$ and the current mission statistics to obtain policy $\pi_t$}
    \EndIf
    \State \changed{Observe the current state $s_t$ and select $a_t=\pi_t(s_t)$}
    \State \changed{Execute $a_t$ and update the exposure and novelty statistics as in Algorithm~\ref{alg:adaptive_robust}}
    \State \changed{Keep $\widebar U_v$ and $U_{s,a}$ unchanged}
\EndFor
\end{algorithmic}
\end{algorithm}

The nominal baseline uses the maximum a \emph{posteriori} (MAP) principle to address uncertainty. The overall procedure of the nominal planner is outlined in \Cref{alg:nominal}. After each observation, the planner infers the most probable \changed{threat prototype} at each node based on the current posterior belief. The planner then selects the next actions, assuming those estimates are correct. This approach reflects an optimistic decision-making strategy that commits to the maximum-probability threat hypothesis under the current posterior. It enables efficient and confident decisions when the belief model is accurate. However, it can become overconfident and fragile in highly uncertain or noisy environments.

\begin{algorithm}[htbp]
\caption{Nominal Planning}
\label{alg:nominal}
\begin{algorithmic}[1]
\State \changed{Initialize $b_{0,v}(\theta)=1/|\Theta|$ for all $v\in\mc V$ and $\theta\in\Theta$}
\For{$t=0,1,\ldots,T-1$}
    \If{$t\equiv0\pmod K$}
        \State \changed{Compute $\widehat\theta_{t,v}\in\arg\max_{\theta\in\Theta}b_{t,v}(\theta)$ for every $v\in\mc V$}
        \State \changed{Apply nominal value iteration using the maximum posterior prototypes and the current mission statistics to obtain policy $\pi_t$}
    \EndIf
    \State \changed{Observe the current state $s_t$ and select $a_t=\pi_t(s_t)$}
    \If{\changed{$s_t=(v_t,\mathrm S)$}}
        \State \changed{Observe $o_{v_t,t}$ and update $b_{t+1,v_t}$ using \eqref{eq:bayesian_belief_update}}
        \State \changed{Keep $b_{t+1,v}=b_{t,v}$ for all $v\neq v_t$}
    \Else
        \State \changed{Keep $b_{t+1,v}=b_{t,v}$ for all $v\in\mc V$}
    \EndIf
    \State \changed{Execute $a_t$ and update the exposure and novelty statistics as in Algorithm~\ref{alg:adaptive_robust}}
\EndFor
\end{algorithmic}
\end{algorithm}

\subsection{Experiment 1: Basic Validation in a Controlled Environment}
This experiment verifies that the proposed adaptive robust planner operates as intended in a simplified and well-controlled setting. We consider a small grid network in which each threat prototype induces Gaussian observation and exposure models, allowing us to confirm that the planner converges to stable policies, correctly updates its ambiguity sets, and reproduces the expected behavior under idealized conditions. This setting conceptually reflects controlled ISR mission test ranges where conditions are intentionally idealized \cite{billingsley2002low}, providing a clean environment for baseline validation before moving to larger and more uncertain scenarios.

The environment is configured as a \(3\times4\) grid graph consisting of 12 nodes. Each node is assigned one of three latent \changed{threat prototypes}, denoted as \(\theta \in \{1, 2, 3\}\). The agent can choose among four sensing actions \(a \in \{A, B, C, D\}\), each producing an observation \(o \sim \mathcal{N}(\mu_{a,\theta}, \sigma_{a,\theta}^2)\) and an associated exposure score that determines exposure risk. \changed{The exposure score follows $z_{v,t}\sim\mathcal N\left(\mu^{\mathrm{haz}}_{a,\theta},\left(\sigma^{\mathrm{haz}}_{a,\theta}\right)^2\right)$, and the exposure indicator is defined as $\eta_{v,t}=\mathds{1}[z_{v,t}>\tau_\eta]$.} The immediate sensing reward combines the observation value, sensing cost, and exposure penalty, while movements between neighboring nodes incur a fixed cost. A discount factor \(\gamma = 0.98\) and additional penalty components model persistent and cumulative exposure effects. The key simulation parameters are summarized in \Cref{tab:exp1_params}.

\begin{table}[htbp]
\centering
\caption{Key simulation parameters for Experiment~1}
\label{tab:exp1_params}
\begin{tabular}{llc}
\toprule
\textbf{Parameter} & \textbf{Description} & \textbf{Value / Setting} \\
\midrule
\changed{Threat prototypes} & $\theta \in \{1, 2, 3\}$ & 3 categories \\
Sensing actions & $a \in \text{\{A, B, C, D\}}$ & 4 actions \\
Observation model &\changed{$o\sim\mathcal N(\mu_{a,\theta},\sigma_{a,\theta}^2)$} & Gaussian \\
Exposure \changed{score} model &
\changed{$z_{v,t}\sim\mathcal N\left(\mu^{\mathrm{haz}}_{a,\theta},(\sigma^{\mathrm{haz}}_{a,\theta})^2\right)$} &
Gaussian\\
Exposure threshold & $\tau_{\eta}$ & 0.5 \\
Discount factor & $\gamma$ & 0.98 \\
Movement cost & $c_{\mathrm{move}}$ & 1.0 \\
Sensing cost & $c_{\mathrm{sense}}(a)$ & $\{1.0,\,1.0,\,1.0,\,0.1\}$ \\
Immediate exposure weight & $\lambda_{\mathrm{imm}}$ & 50.0 \\
Persistent decay rate & $\alpha$ & 0.95 \\
Persistent exposure weight & $\lambda_{\mathrm{pers}}$ & 0.1 \\
Cumulative exposure weight & $\lambda_{\mathrm{cum}}$ & 0.0005 \\
Novelty weight & \changed{$\lambda_{\mathrm{nov}}$} & 1 \\
Novelty decay rate & $\beta$ & 0.8 \\
Replan frequency & \changed{$K$} (steps per replan) & 1 \\
\bottomrule
\end{tabular}
\end{table}

We compare the adaptive robust planner against the static robust and nominal baselines under identical settings. Each planner operates for up to 2000 time steps, replanning at every step ($K=1$). Performance is evaluated based on three key metrics: (i) the average observation reward and cumulative reward over time, (ii) the convergence behavior of the average credible set size $|U_v^t|$ (for the adaptive robust planner only), and (iii) exposure statistics that capture both persistent and cumulative penalty effects. These metrics collectively illustrate how the planner balances reward maximization and risk mitigation as it progressively refines its uncertainty representation. 

All experiments use fixed random seeds for graph generation and threat initialization to ensure reproducibility while maintaining stochastic observation and exposure outcomes. Before presenting quantitative results, we visualize the underlying reward and hazard structures to clarify how the planner perceives the trade-off between sensing benefit and exposure risk in \Cref{fig:reward_distributions,fig:hazard_distributions}. The sensing actions A through D correspond to the four available sensing modes defined in the threat–observation model. As shown, each \changed{threat prototype} favors a distinct action that simultaneously enables more advantageous decisions. Action~D, in contrast, serves as a safer default option when little information about the true \changed{threat prototype} is available. Movement actions allow the aircraft to transition to adjacent nodes in the graph, but they are not depicted here since the figures focus solely on sensing-related reward and hazard patterns.

\begin{figure}[htbp]
    \centering
    \begin{subfigure}[b]{0.33\textwidth}
        \centering
        \includegraphics[width=\textwidth]{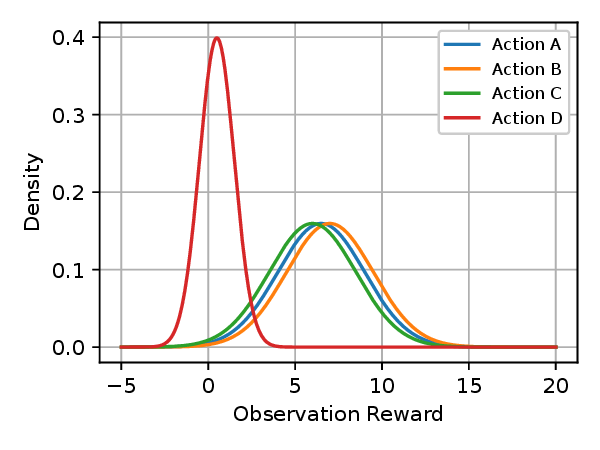}
        \caption{Threat 1}
    \end{subfigure}
    \begin{subfigure}[b]{0.33\textwidth}
        \centering
        \includegraphics[width=\textwidth]{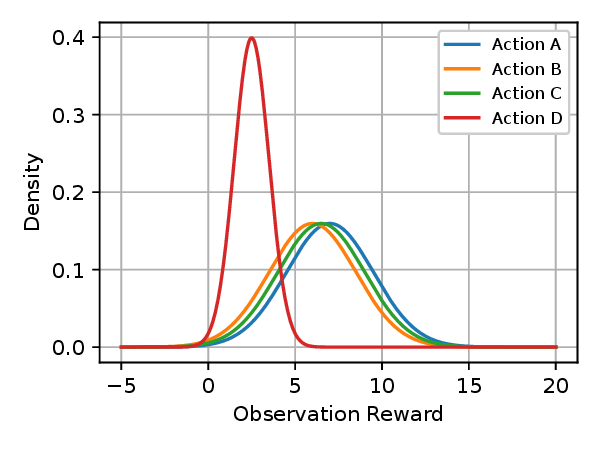}
        \caption{Threat 2}
    \end{subfigure}
    \begin{subfigure}[b]{0.33\textwidth}
        \centering
        \includegraphics[width=\textwidth]{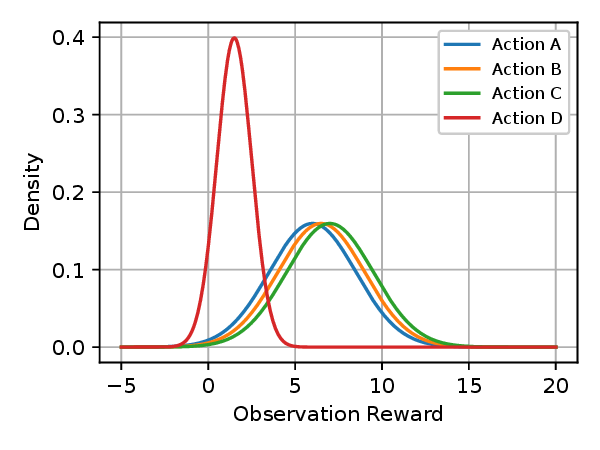}
        \caption{Threat 3}
    \end{subfigure}
    \caption{Comparison of reward distributions across different threat prototypes for Experiment~1.}
    \label{fig:reward_distributions}
\end{figure}

\begin{figure}[htbp]
    \centering
    \begin{subfigure}[b]{0.33\textwidth}
        \centering
        \includegraphics[width=\textwidth]{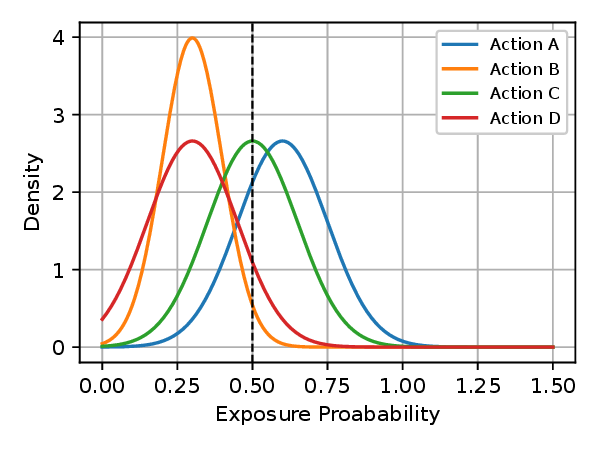}
        \caption{Threat 1}
    \end{subfigure}
    \begin{subfigure}[b]{0.33\textwidth}
        \centering
        \includegraphics[width=\textwidth]{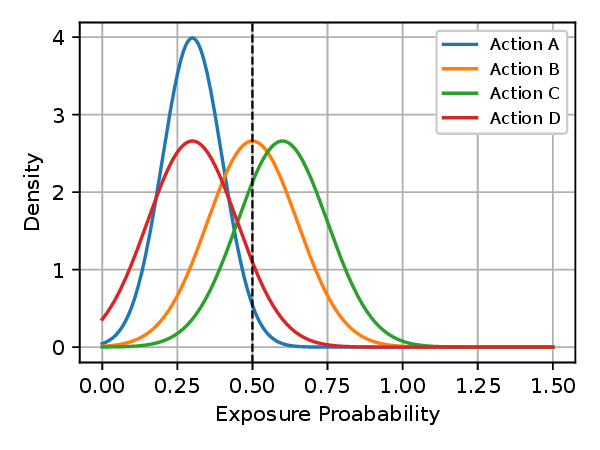}
        \caption{Threat 2}
    \end{subfigure}
    \begin{subfigure}[b]{0.33\textwidth}
        \centering
        \includegraphics[width=\textwidth]{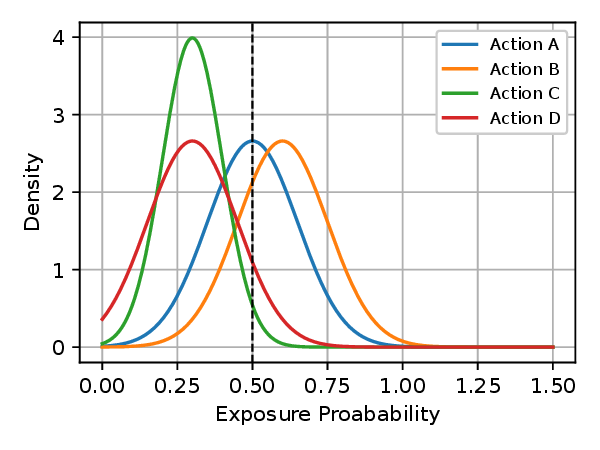}
        \caption{Threat 3}
    \end{subfigure}
    \caption{Exposure probability distributions with dashed line threshold for the threat prototypes in Experiment~1.}
    \label{fig:hazard_distributions}
\end{figure}

\Cref{tab:exp1_summary} presents the results of Experiment~1, where each value is listed as mean~$\pm$~standard deviation across 100 runs. The three planners display clearly different behaviors. The adaptive robust planner performs best overall, achieving high total rewards while keeping the number of exposures low. This means it learns to gather useful information without taking excessive risks. As its belief about threats becomes more confident, the policy naturally shifts from cautious exploration to more efficient sensing and movement. The nominal planner, which assumes a single estimated threat model, gains higher observation rewards early on but suffers from heavy exposure penalties and unstable total returns. Its aggressive sensing leads to frequent exposures, showing that overconfidence under uncertainty can reduce mission survivability. The static robust planner takes the opposite approach: it keeps the ambiguity sets fixed, acts conservatively to avoid risk, but gathers little information and achieves low overall rewards. Together, these results demonstrate that the adaptive robust planner provides the best balance between safety and efficiency. It remains robust when uncertainty is high and approaches nominal-level performance as threat ambiguity at each node decreases.

\begin{table}[htbp]
\centering
\caption{Summary of Experiment~1 Results. Higher values are better for observation and total reward, while lower values are better for cumulative exposures.}
\label{tab:exp1_summary}
\begin{tabular}{lccc}
\toprule
\textbf{Planner} & \textbf{Observation Reward} $\uparrow$ & \textbf{Cumulative Exposures} $\downarrow$ & \textbf{Total Reward} $\uparrow$ \\
\midrule
Adaptive Robust & 5882.44 $\pm$ 168.12 & \textbf{39.13 $\pm$ 11.19} & \textbf{9861.38 $\pm$ 1345.84} \\
Nominal & \textbf{6780.49 $\pm$ 108.32} & 240.92 $\pm$ 65.92 & $-7782.01$ $\pm$ 4887.74 \\
Static Robust & 1487.82 $\pm$ 233.79 & 89.64 $\pm$ 9.16 & 190.63 $\pm$ 858.26 \\
\bottomrule
\end{tabular}
\end{table}

Additionally, we evaluated the computational efficiency of each planner to assess online feasibility. All experiments were conducted on a standard CPU node of the University of Michigan Great Lakes computing cluster, equipped with dual Intel Xeon Gold 6154 processors (36 cores, 3.0 GHz) and 187 GB RAM. Each run used 4 CPU cores and 8 GB of memory without GPU acceleration. For the 2000-step simulations in Experiment~1, the nominal planner required an average of 43.46~$\pm$~0.07~s (approximately 21.7~ms per step), the adaptive robust planner 63.82~$\pm$~0.47~s (31.9~ms per step), and the static robust planner 94.14~$\pm$~0.34~s (47.1~ms per step). All planners operated fast enough for online execution, with each decision step completed within tens of milliseconds. Notably, the adaptive robust planner ran faster than the static robust planner because it gradually narrowed the ambiguity set during execution. As the planner refined its uncertainty over time, the computations became simpler while still supporting real-time operation.

\begin{figure}[htbp]
    \centering
    \begin{subfigure}[b]{0.4\textwidth}
        \centering
        \includegraphics[width=\textwidth]{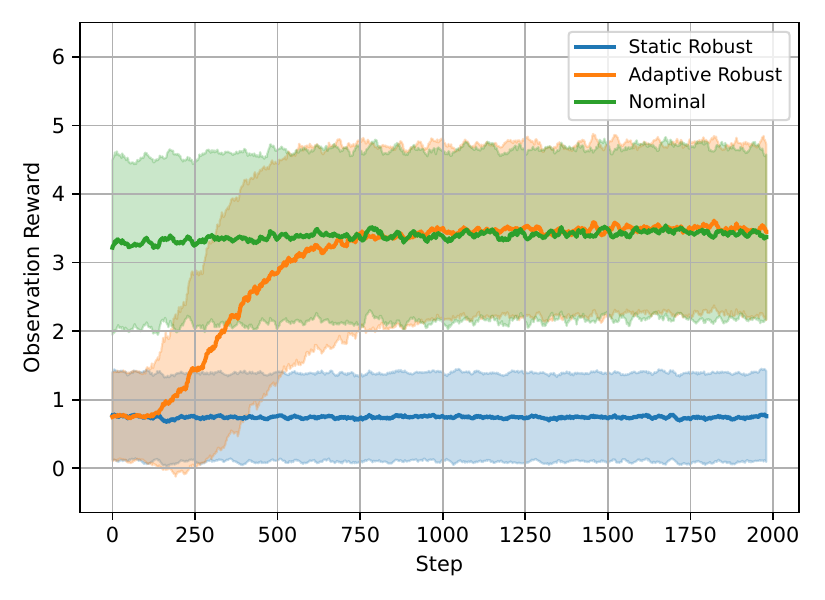}
        \caption{Observation rewards over steps.}
        \label{fig:reward_change_exp1_a}
    \end{subfigure}
    \begin{subfigure}[b]{0.4\textwidth}
        \centering
        \includegraphics[width=\textwidth]{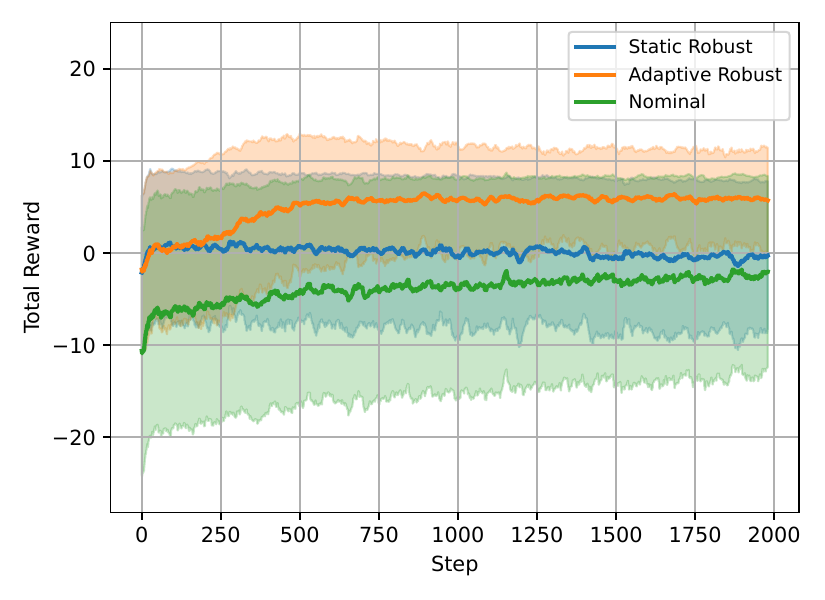}
        \caption{Total rewards over steps.}
        \label{fig:reward_change_exp1_b}
    \end{subfigure}
    \caption{Step-wise observation and total rewards across three planning strategies in Experiment~1.}
    \label{fig:reward_change_exp1}
\end{figure}

\Cref{fig:reward_change_exp1} presents how each planner’s rewards evolve over time. The observation rewards are shown in \Cref{fig:reward_change_exp1_a}. Rewards are plotted as 20-step moving averages with shaded \(\pm\)1 standard deviation. Initially, the nominal planner earns higher rewards because of its optimistic decision-making approach. In contrast, both the adaptive robust and static robust planners begin with lower values due to their conservative strategies under ambiguity sets. Over time, the static robust planner maintains nearly constant rewards because its fixed ambiguity set leads to repetitive behavior under similar conditions. Conversely, the adaptive robust planner gradually shrinks its ambiguity set, increasing observation rewards that approach those achieved by the nominal planner. The observation rewards increase smoothly as each node in the graph maintains its own ambiguity set. Some nodes quickly identify one threat prototype as most likely, whereas others remain ambiguous and continue to consider several possible \changed{threat prototypes}.

\Cref{fig:reward_change_exp1_b} illustrates the evolution of total rewards, which incorporate exposure penalties. The adaptive robust planner consistently achieves the best performance among the three approaches. Although the nominal planner secures higher observation rewards, its exposure risk reduces its total gains. Despite gradual improvement through ongoing belief updates and MAP estimation, the nominal planner exhibits a wider standard deviation band, reflecting unstable performance. In contrast, the adaptive robust planner initially performs similarly to the static robust planner but soon improves as its ambiguity set contracts, achieving high, stable total rewards. Since the static robust planner maintains a larger ambiguity set than the adaptive robust planner at all times, it serves as a lower bound on performance under uncertainty.

\begin{figure}[htbp]
    \centering
    \begin{subfigure}[b]{0.4\textwidth}
        \centering
        \includegraphics[width=\textwidth]{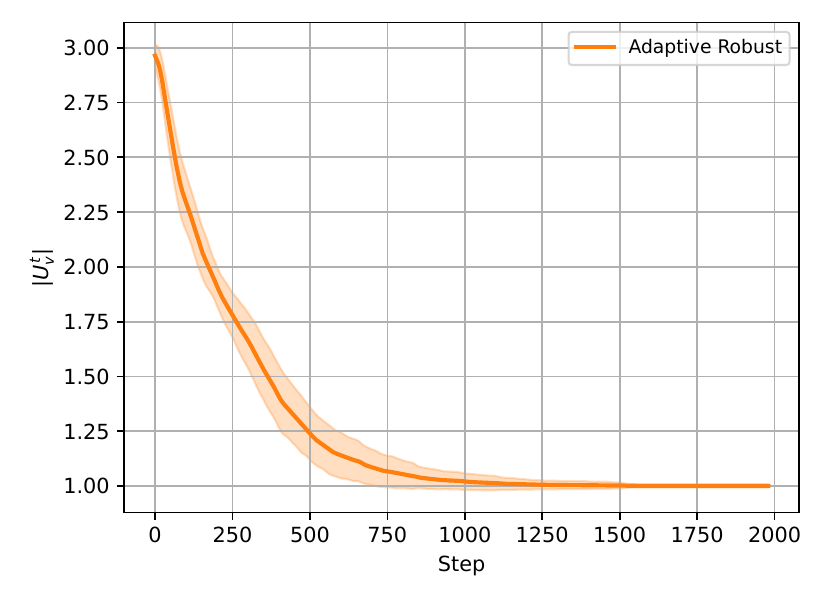}
        \caption{Average credible set size over steps.}
        \label{fig:results_others_exp1_a}
    \end{subfigure}
    \begin{subfigure}[b]{0.4\textwidth}
        \centering
        \includegraphics[width=\textwidth]{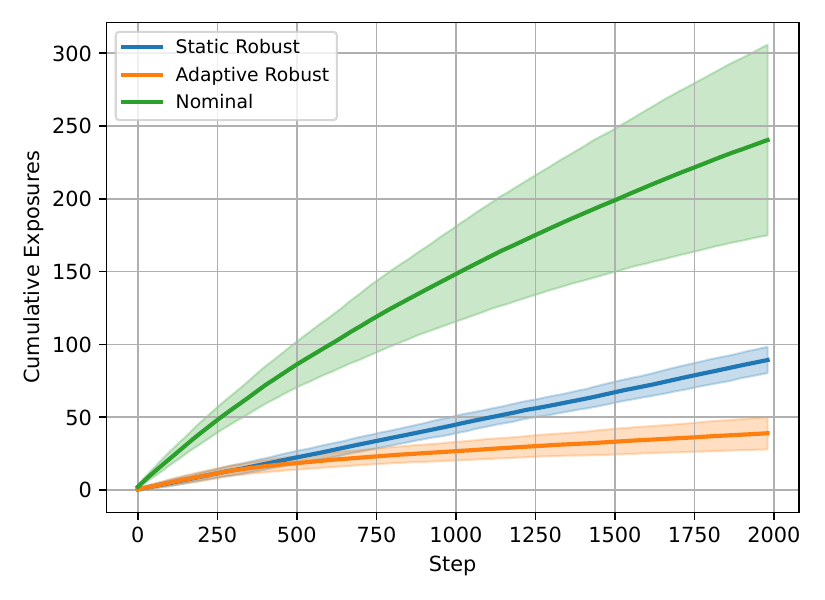}
        \caption{Cumulative exposures over steps.}
        \label{fig:results_others_exp1_b}
    \end{subfigure}
    \caption{Credible set size and cumulative exposures over steps for three planners in Experiment~1.}
    \label{fig:results_others_exp1}
\end{figure}

\Cref{fig:results_others_exp1} illustrates how each planner behaves within the simulation environment beyond the reward metrics. \Cref{fig:results_others_exp1_a} shows the evolution of the average credible set size of the adaptive robust planner across all graph nodes over time. From step~0 to approximately~500, the credible set size decreases sharply, indicating a rapid reduction in uncertainty. Beyond this point, the curve flattens as the planner revisits a few nodes that remain difficult to classify. Eventually, the credible set size converges to 1 for all nodes, indicating that the threats associated with each node are correctly identified.

\Cref{fig:results_others_exp1_b} presents the cumulative exposures over time. The nominal planner exhibits a notably higher exposure rate than the others due to the optimistic policy. The adaptive robust and static robust planners perform similarly up to around step~500. After that, the adaptive robust planner shows lower exposure rates. This trend occurs because our experimental setup rewards planners which accurately identify each node’s threat, giving them an advantage as uncertainty decreases. Notably, the stepwise total reward can exceed the observation reward, as it includes a novelty term that provides a positive exploration bonus for visiting less-explored nodes. Together, these results demonstrate that the adaptive robust planner maintains robustness in uncertain environments while naturally improving efficiency, thereby achieving the intended balance between the two. 

\begin{figure}[htbp]
    \centering
    \includegraphics[width=0.5\linewidth]{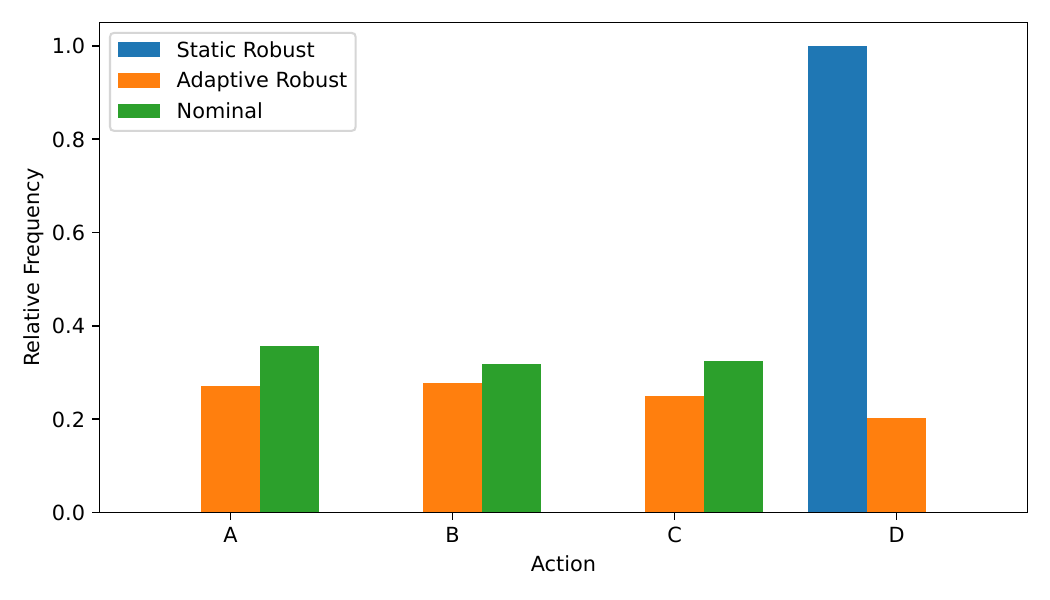}
    \caption{Distribution of sensing action selections across different policies in \changed{Experiment~1}.}
    \label{fig:exp1_action}
\end{figure}

\Cref{fig:exp1_action} illustrates how each planner selected sensing actions in Experiment~1, clearly revealing their distinct decision-making strategies. Action labels (A–D) denote the mapped sensing categories, and the reported frequencies are normalized to reflect relative selection ratios. The static robust planner consistently chooses only action D, the safest option when the \changed{threat prototype} is uncertain. This conservative behavior results in relatively low observation and total rewards. In contrast, the nominal planner, which relies on the MAP estimate, avoids action D and selects actions A, B, and C depending on its belief about each node’s \changed{threat prototype}. The adaptive robust planner combined the strengths of both approaches, showing a balanced distribution across all four actions (A-D). It initially behaved conservatively, like the static robust planner, but gradually shifted toward higher-reward actions as the \changed{threat prototype} became more certain.

\subsection{Experiment 2: Generalization to Non-Gaussian Threat Models}

We assess the planner’s robustness and adaptability under general threat models, including Gaussian mixtures and log-normal distributions. These distributions capture multimodal and heavy-tailed sensing patterns commonly observed in electronic warfare environments \cite{ward2013seaclutter}, allowing us to test whether the planner generalizes beyond the idealized Gaussian assumptions of Experiment~1. The modified experimental parameters are summarized in \Cref{tab:exp2_modified_final}. With the same 3×4 grid graph and the same action set as in Experiment~1, we modify the observation and exposure models to capture more diverse threat prototypes. The reward and exposure models used in Experiment~2 for all three threats are shown in \Cref{fig:reward_distributions_2,fig:hazard_distributions_2}. Specifically, Threats~1 and~2 use Gaussian mixture observation models, where a single Gaussian is treated as a special (degenerate) case, while Threat~3 adopts a log-normal observation model. Each run consists of 3000 planning steps and is simulated under 100 independent random seeds. We evaluate whether the adaptive robust planner remains effective under non-Gaussian uncertainty, exhibiting similar patterns in ambiguity refinement and exposure management to those observed in Experiment~1.

\begin{table}[htbp]
\centering
\caption{Modified parameters in Experiment~2}
\label{tab:exp2_modified_final}
\begin{tabular}{lll}
\toprule
\textbf{Parameter} & \textbf{Description} & \textbf{New Value / Setting} \\
\midrule
Observation model & $o \sim$ Gaussian mixture or log-normal mixture & Mixed Gaussian / log-normal \\
Exposure \changed{score} model &
\changed{$z\sim$ Gaussian mixture} &
Mixed Gaussian \\
Sensing cost & $c_{\mathrm{sense}}(a)$ & $\{0.1, 0.1, 0.1, 0.1\}$ \\
Immediate exposure \changed{weight} &
$\lambda_{\mathrm{imm}}$ &
\changed{$10.0$} \\
Persistent decay rate & $\alpha$ & $0.99$ \\
Persistent exposure weight & $\lambda_{\mathrm{pers}}$ & $0.5$ \\
Cumulative exposure weight & $\lambda_{\mathrm{cum}}$ & $0.01$ \\
\bottomrule
\end{tabular}
\end{table}

\begin{figure}[htbp]
    \centering
    \begin{subfigure}[b]{0.33\textwidth}
        \centering
        \includegraphics[width=\textwidth]{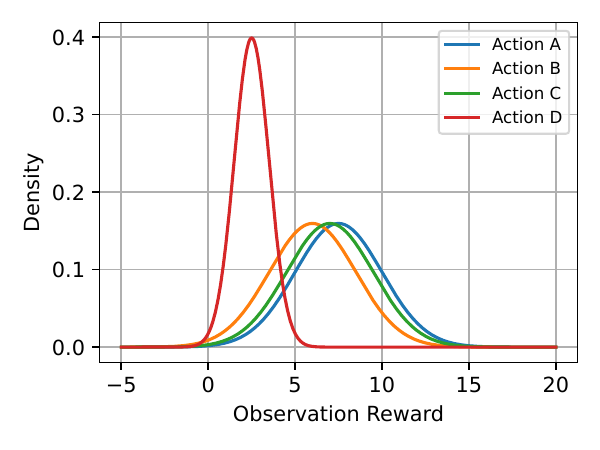}
        \caption{Threat 1}
    \end{subfigure}
    \begin{subfigure}[b]{0.33\textwidth}
        \centering
        \includegraphics[width=\textwidth]{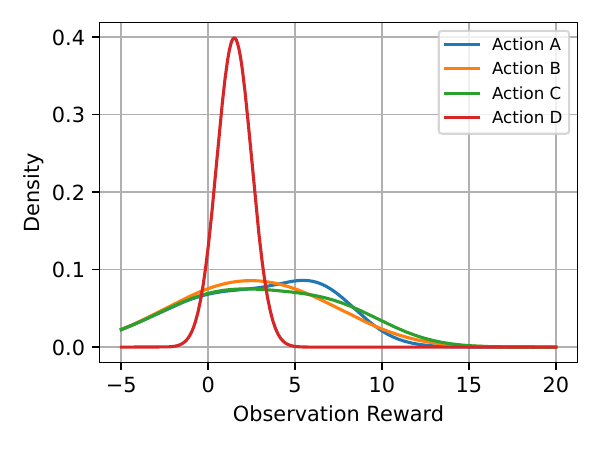}
        \caption{Threat 2}
    \end{subfigure}
    \begin{subfigure}[b]{0.33\textwidth}
        \centering
        \includegraphics[width=\textwidth]{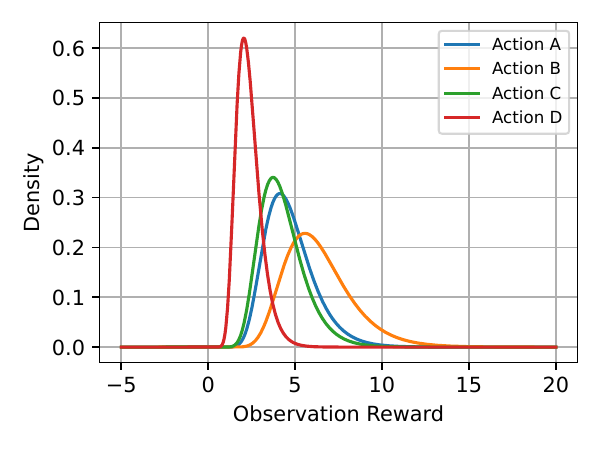}
        \caption{Threat 3}
    \end{subfigure}
\caption{\changed{Comparison of reward distributions across different threat prototypes for Experiment~2.}}    
    \label{fig:reward_distributions_2}
\end{figure}

\begin{figure}[htbp]
    \centering
    \begin{subfigure}[b]{0.33\textwidth}
        \centering
        \includegraphics[width=\textwidth]{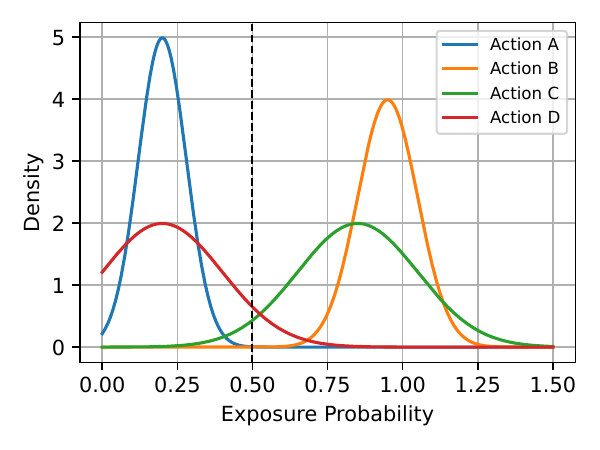}
        \caption{Threat 1}
    \end{subfigure}
    \begin{subfigure}[b]{0.33\textwidth}
        \centering
        \includegraphics[width=\textwidth]{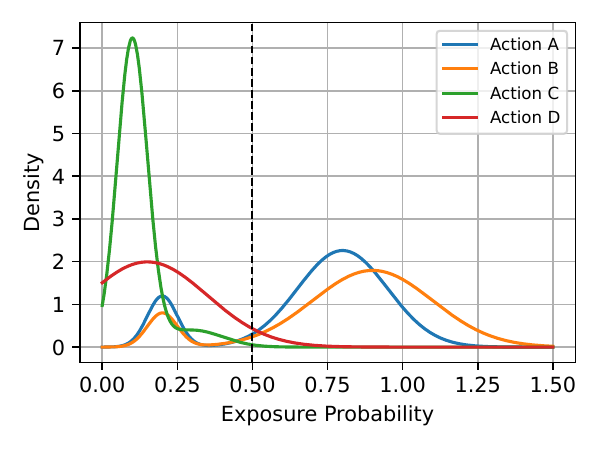}
        \caption{Threat 2}
    \end{subfigure}
    \begin{subfigure}[b]{0.33\textwidth}
        \centering
        \includegraphics[width=\textwidth]{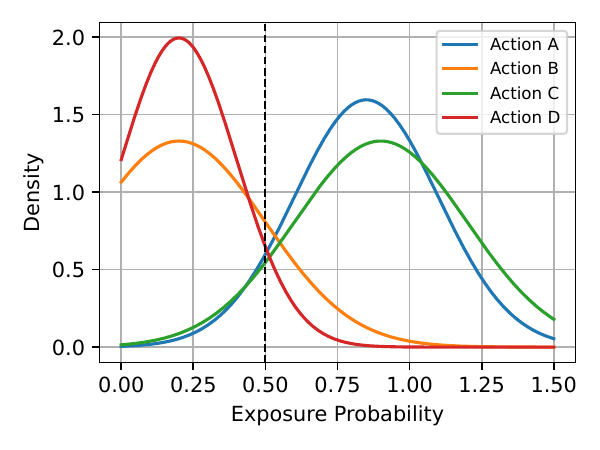}
        \caption{Threat 3}
    \end{subfigure}
    \caption{Exposure probability distributions with dashed line threshold for the threat prototypes in \changed{Experiment~2}.}
    \label{fig:hazard_distributions_2}
\end{figure}

\Cref{tab:exp2_summary} summarizes the quantitative results obtained under modified threat models. Despite the changes in both the observation and exposure distributions, the overall performance trend remains consistent with that observed in Experiment~1. The adaptive robust planner continues to achieve substantially higher total rewards compared to the nominal and static robust baselines, while continuing to keep cumulative exposures low. This consistency indicates that the proposed adaptive robust approach remains effective beyond Gaussian threat assumptions, generalizing across broader classes of uncertainty and confirming the robustness of the planning method.

\begin{table}[htbp]
\centering
\caption{Summary of Experiment~2 Results.}
\label{tab:exp2_summary}
\begin{tabular}{lccc}
\toprule
\textbf{Planner} & \textbf{Observation Reward} $\uparrow$ & \textbf{Cumulative Exposures} $\downarrow$ & \textbf{Total Reward} $\uparrow$ \\
\midrule
Adaptive Robust & 7334.50 $\pm$ 936.78 & \textbf{77.29 $\pm$ 26.41} & \textbf{4416.98 $\pm$ 3101.26} \\
Nominal & \textbf{8414.99 $\pm$ 1069.28} & 109.11 $\pm$ 39.93 & $-153.79$ $\pm$ 5352.14 \\
Static Robust & 3210.33 $\pm$ 297.09 & 87.62 $\pm$ 11.74 & $-1998.17$ $\pm$ 1505.94 \\
\bottomrule
\end{tabular}
\end{table}

\begin{figure}[htbp]
    \centering
    \begin{subfigure}[b]{0.4\textwidth}
        \centering
        \includegraphics[width=\textwidth]{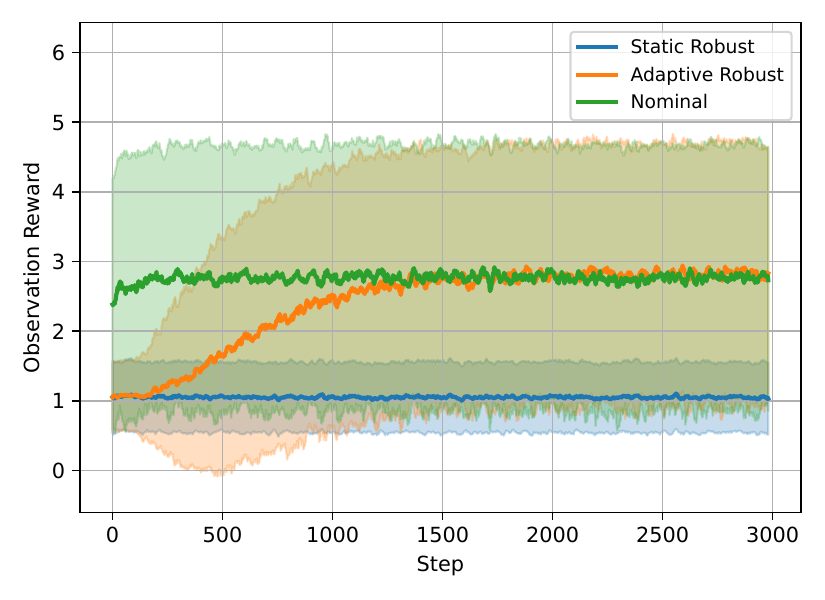}
        \caption{Observation rewards over steps.}
        \label{fig:reward_change_exp2_a}
    \end{subfigure}
    \begin{subfigure}[b]{0.4\textwidth}
        \centering
        \includegraphics[width=\textwidth]{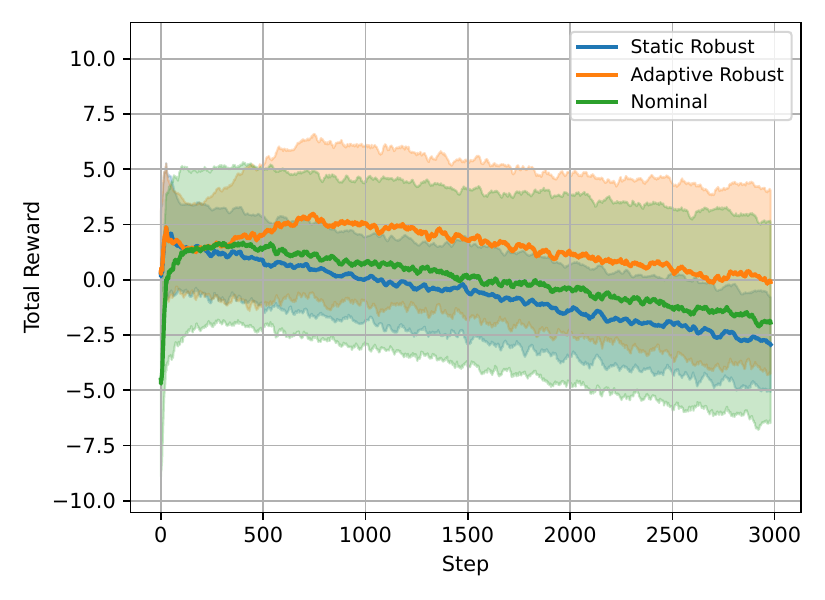}
        \caption{Total rewards over steps.}
        \label{fig:reward_change_exp2_b}
    \end{subfigure}
    \caption{Step-wise observation and total rewards across three planning strategies in Experiment~2.}
    \label{fig:reward_change_exp2}
\end{figure}

\Cref{fig:reward_change_exp2} illustrates the temporal evolution of both observation and total rewards across the three planning strategies. Reward trajectories are shown as 20-step moving averages with shaded regions indicating \(\pm\)1 standard deviation. Similar to the Gaussian case in Experiment~1, the adaptive robust planner consistently maintains higher observation rewards while stabilizing total returns through balanced exploration and exposure control. In contrast, the nominal planner achieves higher observation rewards but yields lower overall performance under uncertain threat conditions. \Cref{fig:reward_change_exp2_b} shows a gradual decrease in total rewards across all methods compared to Experiment~1. This trend results from the increased cumulative exposure weight $\lambda_{\mathrm{cum}}$, which amplifies long-term penalty accumulation. Regardless of the parameter changes, the adaptive robust planner still outperforms the others and gradually recovers its total rewards as the ambiguity set \changed{contracts} over time.

\begin{figure}[htbp]
    \centering
    \begin{subfigure}[b]{0.4\textwidth}
        \centering
        \includegraphics[width=\textwidth]{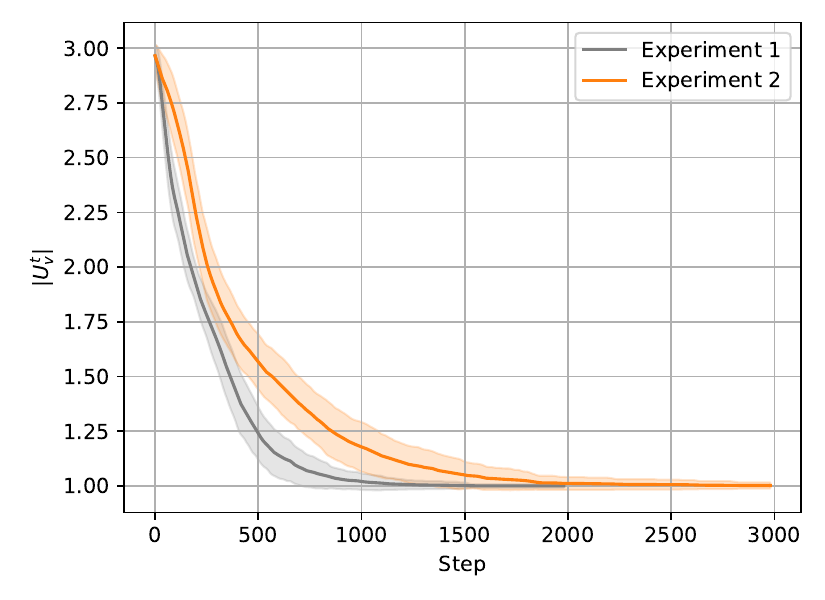}
        \caption{Average credible set size over steps.}
        \label{fig:results_others_exp2_a}
    \end{subfigure}
    \begin{subfigure}[b]{0.4\textwidth}
        \centering
        \includegraphics[width=\textwidth]{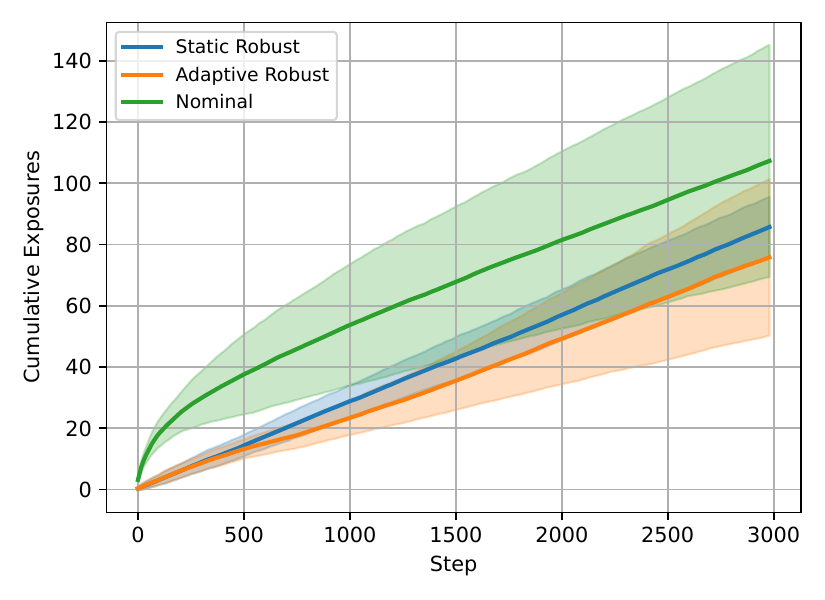}
        \caption{Cumulative exposures over steps.}
        \label{fig:results_others_exp2_b}
    \end{subfigure}
    \caption{Credible set size and cumulative exposures over steps for three planners in Experiment~2.}
    \label{fig:results_others_exp2}
\end{figure}

The underlying adaptive behavior is further analyzed in \Cref{fig:results_others_exp2}. Both trajectories show 20-step moving averages with shaded regions indicating one standard deviation, where the adaptive planner exhibits lower exposure despite increased variability. \Cref{fig:results_others_exp2_a} shows that the credible set size $|U_v^t|$ steadily decreases as the adaptive planner gathers more informative observations, confirming effective uncertainty reduction under non-Gaussian distributions. Despite having the same graph size as in Experiment~1, the reduction progresses more slowly due to the broader coverage of the Gaussian mixture models, which makes distinguishing among threats more challenging. Correspondingly, \Cref{fig:results_others_exp2_b} demonstrates that this reduction in uncertainty translates to lower cumulative exposure compared to the nominal and static robust baselines. The overall variance in both rewards and exposures is higher than in Experiment~1, reflecting the wider spread of the observation and threat models used in Experiment~2.

\begin{figure}[htbp]
    \centering
    \includegraphics[width=0.5\linewidth]{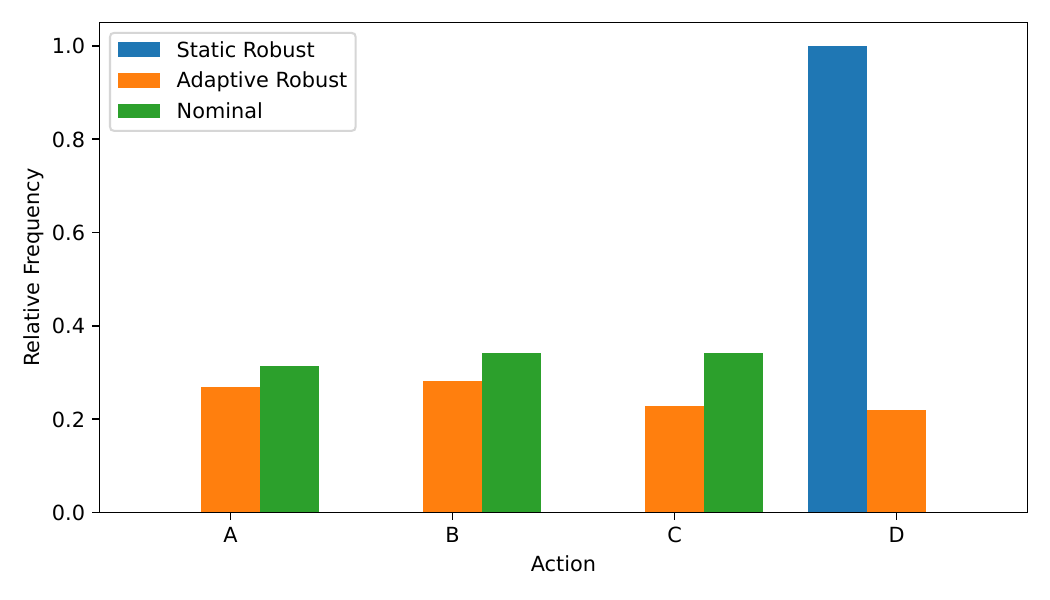}
    \caption{Distribution of sensing action selections across different policies in Experiment~2.}
    \label{fig:exp2_action}
\end{figure}

Finally, \Cref{fig:exp2_action} compares the action selection distributions across the three planners, showing patterns similar to those in Experiment~1. Actions A–D represent the sensing categories, with normalized selection frequencies. The static robust policy consistently favors the safer Action~D under uncertainty, while the adaptive robust planner initially exhibits similar conservative behavior but gradually shifts toward a more balanced selection as uncertainty decreases. These results demonstrate that the adaptive robust planner effectively adjusts its sensing strategy while maintaining robustness across different threat conditions, even when the underlying threat models are non-Gaussian.

\subsection{Experiment 3: Scalability across Network Topologies}
To examine scalability, we test how the planner's performance and computation time change with increasing network size and structural complexity. We test a variety of graph topologies that reflect practical ISR mission environments, ranging from regular grid-like layouts to irregular, hub-dominated, and modular regions. Assessing performance across these heterogeneous topologies allows us to measure robustness under diverse operational environments. 

\textit{Regular grid networks} (\(3\times4\), \(5\times5\), and \(6\times6\)) have a simple and regular structure, which allows direct comparison with the controlled setting of Experiment~1. \textit{Grids with random edge deletions} (\(5\times5\) and \(6\times6\)) introduce local connectivity disruptions, reflecting, e.g., terrain occlusions or restricted flight corridors \cite{10.1007/978-3-642-22993-0_49,SEYMOUR1978227}. \textit{Star networks}, composed of a central hub and 8 or 14 leaf nodes, are included to examine how the planner behaves in highly centralized hub-and-spoke structures. \textit{Erdos–Renyi random graphs} with 15 and 30 nodes (\(p = 0.10, 0.15\)) offer uniform random connectivity, serving as a baseline for comparison in networks without specific structural patterns. \textit{Barabasi–Albert scale-free graphs} of comparable size (\(m = 2, 3\)) represent networks where a few hubs have many connections. Finally, \textit{stochastic block models (community)} with 30 and 40 nodes, containing 3 and 4 clusters (\(p_{\text{intra}} = 0.7\), \(p_{\text{inter}} = 0.05\) or \(0.03\)), are used to test how well the planner adapts to modular networks with clear community structures.

Each experiment is repeated across 10 random seeds for each graph configuration and planner type (adaptive robust, static robust, and nominal) to ensure statistical consistency. All planners were tested in the same simulation environment as in Experiment~1. The exposure and reward settings were identical, but the simulation horizon was extended to 8000 time steps. Random seeds are fixed for graph generation and threat setup to ensure reproducibility while allowing for randomness in observations and exposures.

\Cref{tab:exp3_summary} summarizes the simulation results of Experiment~3, where all values are reported as mean~$\pm$~standard deviation across 10 random seeds. As shown in the table, the adaptive robust planner consistently achieved higher total rewards than both the nominal and static robust planners across all network topologies. For the observation rewards, values were generally similar within the same planner, regardless of the underlying graph type or size. This consistency arises because all simulations were run for 8000 steps, providing comparable opportunities for each planner to accumulate observable rewards. The nominal planner showed substantially higher cumulative exposure, indicating more aggressive and dangerous behavior than the other planners. As a result, its higher observation rewards are offset in the total reward value. This pattern is consistent with the findings from \Cref{fig:reward_change_exp1}, where the adaptive robust planner demonstrated superior long-term performance by maintaining a balanced trade-off between safety and efficiency. The results show that the total reward also tended to increase with graph size. This trend occurs because, in larger networks, the novelty term remains high for longer as more nodes become available to explore. Furthermore, the adaptive robust planner exhibited smaller variance in both total rewards and cumulative exposures, confirming its stable and reliable performance across networks with different structures and scales.

\begin{table}[htbp]
\centering
\caption{Summary of Experiment~3 Results.} 
\label{tab:exp3_summary}
\begin{tabular}{llllll}
\toprule
\textbf{Graph Type} & \textbf{Nodes} & \textbf{Planner} & \textbf{Obs. Reward} $\uparrow$ & \textbf{Cum. Exposure} $\downarrow$ & \textbf{Total Reward} $\uparrow$ \\
\midrule

\multirow{6}{*}{Barabasi-Albert} 
& \multirow{3}{*}{15} 
 & Adaptive Robust & 26604.09 ± 228.80 & \textbf{109.60 ± 7.17} & \textbf{64112.19 ± 1991.10} \\
 & & Nominal & \textbf{27633.34 ± 215.70} & 474.80 ± 103.36 & 9686.63 ± 14999.06 \\
 & & Static Robust & 6376.57 ± 676.51 & 365.00 ± 13.33 & 10969.77 ± 2855.34 \\
\cmidrule(lr){2-6}

& \multirow{3}{*}{30} 
 & Adaptive Robust & 25267.94 ± 237.34 & \textbf{121.60 ± 4.25} & \textbf{155323.85 ± 1247.38} \\
 & & Nominal & \textbf{27355.98 ± 151.74} & 744.70 ± 123.64 & 70484.12 ± 13970.15 \\
 & & Static Robust & 6234.28 ± 744.20 & 365.70 ± 14.01 & 105826.26 ± 2308.49 \\
\midrule

\multirow{6}{*}{\begin{tabular}{@{}c@{}}
Stochastic\\
Block Model
\end{tabular}}
& \multirow{3}{*}{30}
 & Adaptive Robust & 25350.52 ± 361.47 & \textbf{121.10 ± 5.90} & \textbf{146729.40 ± 1296.92} \\
 & & Nominal & \textbf{27347.23 ± 187.65} & 744.50 ± 158.72 & 59647.03 ± 19560.29 \\
 & & Static Robust & 6182.58 ± 773.15 & 365.70 ± 14.01 & 96770.31 ± 4566.13 \\
\cmidrule(lr){2-6}

& \multirow{3}{*}{40}
 & Adaptive Robust & 24294.14 ± 320.70 & \textbf{136.60 ± 6.96} & \textbf{218087.10 ± 1372.49} \\
 & & Nominal & \textbf{27181.90 ± 150.27} & 930.70 ± 119.53 & 113903.51 ± 13890.50 \\
 & & Static Robust & 6187.70 ± 568.42 & 365.90 ± 12.64 & 172227.88 ± 2468.59 \\
\midrule

\multirow{6}{*}{Erdos-Renyi}
& \multirow{3}{*}{15}
 & Adaptive Robust & 26944.90 ± 567.93 & \textbf{116.80 ± 35.00} & \textbf{44870.10 ± 21794.13} \\
 & & Nominal & \textbf{27699.09 ± 273.14} & 435.60 ± 169.30 & -3554.78 ± 15704.17 \\
 & & Static Robust & 6245.47 ± 691.76 & 365.00 ± 13.33 & -7520.53 ± 22919.71 \\
\cmidrule(lr){2-6}

& \multirow{3}{*}{30}
 & Adaptive Robust & 25589.12 ± 869.56 & \textbf{119.00 ± 9.63} & \textbf{148749.01 ± 51479.06} \\
 & & Nominal & \textbf{27385.43 ± 255.58} & 720.50 ± 220.94 & 67701.03 ± 36151.77 \\
 & & Static Robust & 6134.88 ± 819.49 & 365.70 ± 14.01 & 98663.02 ± 52665.88 \\
\midrule

\multirow{9}{*}{Grid}
& \multirow{3}{*}{12}
 & Adaptive Robust & 26999.01 ± 196.12 & \textbf{102.90 ± 7.71} & \textbf{44345.51 ± 2482.13} \\
 & & Nominal & \textbf{27689.68 ± 163.95} & 441.90 ± 160.51 & -10734.79 ± 22350.71 \\
 & & Static Robust & 6127.44 ± 815.71 & 364.80 ± 13.06 & -13145.22 ± 1443.97 \\
\cmidrule(lr){2-6}

& \multirow{3}{*}{25}
 & Adaptive Robust & 25687.31 ± 322.45 & \textbf{128.60 ± 22.19} & \textbf{121098.12 ± 5810.31} \\
 & & Nominal & \textbf{27441.00 ± 159.42} & 702.50 ± 172.20 & 34449.63 ± 20104.99 \\
 & & Static Robust & 6184.44 ± 790.95 & 364.50 ± 14.00 & 64835.45 ± 2247.67 \\
\cmidrule(lr){2-6}

& \multirow{3}{*}{36}
 & Adaptive Robust & 24709.20 ± 333.65 & \textbf{151.50 ± 34.08} & \textbf{185971.11 ± 4331.75} \\
 & & Nominal & \textbf{27251.40 ± 180.25} & 861.80 ± 133.23 & 85872.47 ± 15090.80 \\
 & & Static Robust & 6093.93 ± 724.04 & 366.10 ± 12.60 & 134456.16 ± 2058.68 \\
\midrule

\multirow{6}{*}{\begin{tabular}{@{}c@{}}
Grid with\\
Edge 
Deletions
\end{tabular}}
& \multirow{3}{*}{25}
 & Adaptive Robust & 25741.21 ± 145.88 & \textbf{117.70 ± 9.24} & \textbf{132262.88 ± 7148.92} \\
 & & Nominal & \textbf{27466.33 ± 225.79} & 694.70 ± 209.91 & 50786.09 ± 25616.49 \\
 & & Static Robust & 6044.33 ± 730.08 & 368.10 ± 13.55 & 80153.59 ± 8032.60 \\
\cmidrule(lr){2-6}

& \multirow{3}{*}{36}
 & Adaptive Robust & 25134.64 ± 1020.19 & \textbf{126.20 ± 13.41} & \textbf{194330.90 ± 68213.77} \\
 & & Nominal & \textbf{27359.72 ± 283.59} & 683.50 ± 217.12 & 118948.59 ± 52728.24 \\
 & & Static Robust & 5794.11 ± 1462.13 & 367.60 ± 14.75 & 144842.62 ± 72473.51 \\
\midrule

\multirow{6}{*}{Star}
& \multirow{3}{*}{9}
 & Adaptive Robust & 27198.16 ± 207.06 & \textbf{103.20 ± 8.48} & \textbf{32855.09 ± 1494.85} \\
 & & Nominal & \textbf{27711.70 ± 239.62} & 416.30 ± 144.04 & -16254.38 ± 20902.36 \\
 & & Static Robust & 5593.48 ± 1759.26 & 364.60 ± 13.28 & -22579.18 ± 1664.92 \\
\cmidrule(lr){2-6}

& \multirow{3}{*}{15}
 & Adaptive Robust & 26628.59 ± 263.27 & \textbf{109.80 ± 9.41} & \textbf{67126.14 ± 1709.23} \\
 & & Nominal & \textbf{27689.59 ± 232.86} & 401.50 ± 135.57 & 22525.42 ± 17940.09 \\
 & & Static Robust & 6251.57 ± 1343.95 & 365.00 ± 13.33 & 14061.61 ± 1853.36 \\
\bottomrule

\end{tabular}
\end{table}

Across all tested graph topologies and node sizes, the adaptive robust planner demonstrated stable and consistent convergence, which occurs when all credible sets become singletons (e.g., $|U_v^t|=1$ for the reachable nodes). In a few runs, certain nodes failed to converge due to structural limitations in the network rather than a learning issue with the planner. These cases occurred when the graphs contained disconnected components, making some nodes unreachable from the main connected component. Because the graphs were randomly generated, such disconnected components appeared occasionally. After excluding structurally invalid cases, our adaptive robust planner successfully converged for all reachable nodes. The mean and standard deviation of the convergence steps increased gradually with graph size, indicating that the planner remains reliable across networks of different structures and scales.

\begin{figure}[htbp]
    \centering
    \includegraphics[width=0.6\linewidth]{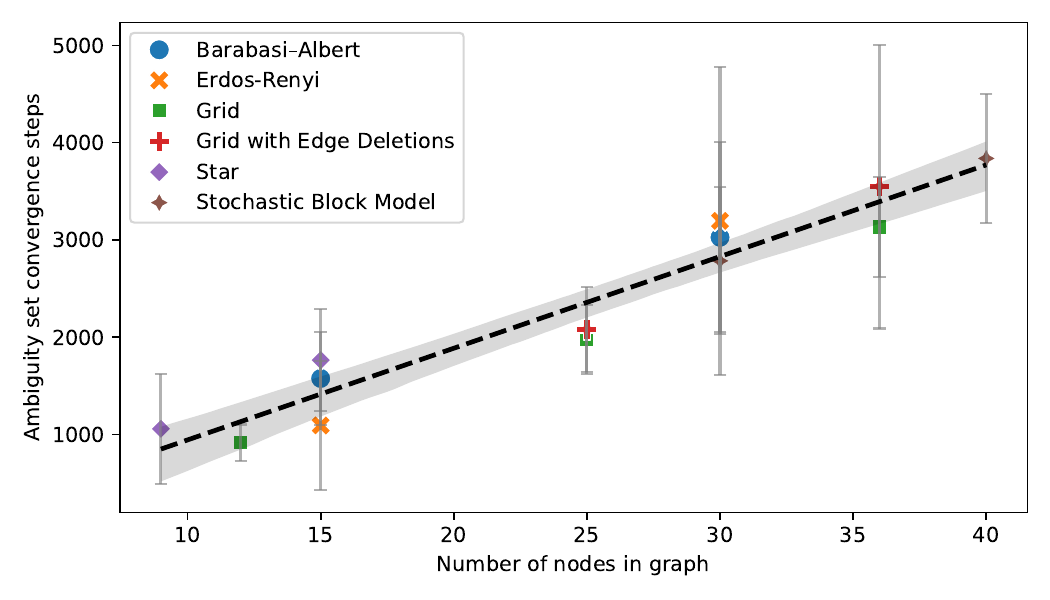}
    \caption{Mean convergence steps of the ambiguity set versus graph size across different network topologies.}
    \label{fig:exp2_convergence}
\end{figure}

Figure~\ref{fig:exp2_convergence} illustrates how graph size and topology influence the performance of the adaptive robust planner. Error bars indicate one standard deviation over random seeds, and the dashed black line denotes the overall linear trend across all network types. Distinct marker shapes represent different graph topologies, while the horizontal axis denotes the number of nodes in each graph. Each point corresponds to the number of steps required for the size of all credible sets to reach one ($|U_v^t|=1$) among the reachable nodes. The results indicate that, regardless of the network topology, the number of nodes is a key parameter that determines the adaptive planner's efficiency. A clear linear relationship is observed between graph size and convergence steps, demonstrating that the proposed adaptive planner operates consistently across diverse graph structures. Each episode was limited to 8000 steps, yet even the slowest-converging case required only around 5200 steps to fully resolve all ambiguity sets.

\section{Conclusion}\label{sec:conclusion}

This work establishes that adaptive robust decision-making enables autonomous ISR agents to systematically reduce conservatism while \changed{retaining a conservative value guarantee under adversarial uncertainty}. By \changed{updating node-wise posterior beliefs and contracting the corresponding credible sets and induced ambiguity sets using online sensing observations}, the proposed framework \changed{gradually reduces the range of threat prototypes considered during robust planning}. \changed{Under the stated assumptions, the robust and nominal Bellman operators and their fixed points agree when the credible sets contract to singleton sets containing the true threat prototypes. Before this contraction is complete, the robust value provides a lower bound on expected performance whenever the true prototypes remain within the credible sets.} These results demonstrate that robustness and adaptivity can be reconciled within a unified sequential \changed{decision-making} framework, allowing surveillance performance to improve as threat uncertainty is resolved while \changed{maintaining conservative planning under the remaining uncertainty}.

Across varying threat distributions and operational graph structures, the adaptive robust planner consistently achieves higher cumulative mission rewards and lower exposure risk than both static robust and nominal planners. \changed{The numerical results also show that credible set contraction reduces the computational burden relative to static robust planning while preserving online feasibility.} These findings \changed{support the effectiveness of combining Bayesian belief refinement, ambiguity contraction, and robust planning to improve the efficiency and reliability of autonomous surveillance in contested environments.}

\appendix
\section{Appendix: Proofs of \Cref{sec:theoretical_analysis}}\label{sec:proof}

\begin{proof}[Proof of Theorem~\ref{thm:as_convergence}]
\changed{Fix a planning time $t$. By Assumptions~\ref{ass:disc} and~\ref{ass:rect}, the robust operator $\mc T_{\mathrm r}^t$ is a $\gamma$ contraction on $(\mathbb R^{|\mc S|},\|\cdot\|_\infty)$. The nominal operator $\mc T_{\mathrm n}^t$ is also a $\gamma$ contraction because its transition kernel is stochastic and its stage reward does not depend on the value function. Therefore, both operators admit unique fixed points $V_{\mathrm r}^{t,\star}$ and $V_{\mathrm n}^{t,\star}$.}

\changed{The reward bound in Assumption~\ref{ass:disc} implies}
\begin{equation}
\changed{
\left\|V_{\mathrm r}^{t,\star}\right\|_\infty
\le
\frac{R_t}{1-\gamma}
=
M_t,
\qquad
\left\|V_{\mathrm n}^{t,\star}\right\|_\infty
\le
\frac{R_t}{1-\gamma}
=
M_t.
}
\label{eq:fixed_point_bounds}
\end{equation}

\changed{For any bounded value function $V$, define}
\begin{equation}
\changed{
G_t(s,a,\theta;V)
=
r^t(s,a;\theta)
+
\gamma
\sum_{s'\in\mc S}
P(s'\mid s,a;\theta)V(s').
}
\label{eq:local_backup_function}
\end{equation}
\changed{For a state $s$ associated with node $v(s)$, the robust and nominal operators can be written as}
\begin{equation}
\changed{
(\mc T_{\mathrm r}^tV)(s)
=
\max_{a\in\mc A(s)}
\min_{\theta\in U_{v(s)}^t}
G_t(s,a,\theta;V)
}
\end{equation}
\changed{and}
\begin{equation}
\changed{
(\mc T_{\mathrm n}^tV)(s)
=
\max_{a\in\mc A(s)}
G_t(s,a,\widehat\theta_{t,v(s)};V).
}
\end{equation}

\changed{Since $\widehat\theta_{t,v(s)}\in U_{v(s)}^t$, for every feasible action $a\in\mc A(s)$,}
\begin{equation}
\changed{
\min_{\theta\in U_{v(s)}^t}
G_t(s,a,\theta;V)
\le
G_t(s,a,\widehat\theta_{t,v(s)};V).
}
\end{equation}
\changed{Moreover, for every $\theta\in U_{v(s)}^t$,}
\begin{align}
&\changed{
\left|
G_t(s,a,\theta;V)
-
G_t(s,a,\widehat\theta_{t,v(s)};V)
\right|
}
\nonumber\\
&\changed{
\le
\left|
r^t(s,a;\theta)
-
r^t(s,a;\widehat\theta_{t,v(s)})
\right|
}
\nonumber\\
&\changed{
\quad+
\gamma
\left|
\sum_{s'\in\mc S}
\left[
P(s'\mid s,a;\theta)
-
P(s'\mid s,a;\widehat\theta_{t,v(s)})
\right]
V(s')
\right|
}
\nonumber\\
&\changed{
\le
\Delta_{r,t}
+
\gamma
\|V\|_\infty
\Delta_{P,t}.
}
\label{eq:local_operator_discrepancy}
\end{align}
\changed{The final inequality follows from}
\begin{equation}
\changed{
\left|
\sum_{s'\in\mc S}
\left[
P(s'\mid s,a;\theta)
-
P(s'\mid s,a;\widehat\theta_{t,v(s)})
\right]
V(s')
\right|
\le
\|V\|_\infty
\left\|
P(\cdot\mid s,a;\theta)
-
P(\cdot\mid s,a;\widehat\theta_{t,v(s)})
\right\|_1.
}
\end{equation}

\changed{Taking the inner minimum and then the outer maximum gives}
\begin{equation}
\changed{
0
\le
(\mc T_{\mathrm n}^tV)(s)
-
(\mc T_{\mathrm r}^tV)(s)
\le
\Delta_{r,t}
+
\gamma
\|V\|_\infty
\Delta_{P,t}.
}
\end{equation}
\changed{Therefore,}
\begin{equation}
\changed{
\left\|
\mc T_{\mathrm n}^tV
-
\mc T_{\mathrm r}^tV
\right\|_\infty
\le
\Delta_{r,t}
+
\gamma
\|V\|_\infty
\Delta_{P,t}.
}
\label{eq:operator_discrepancy_bound}
\end{equation}

\changed{Using the fixed point identities and adding and subtracting $\mc T_{\mathrm r}^tV_{\mathrm n}^{t,\star}$ yields}
\begin{align}
&\changed{
\left\|
V_{\mathrm r}^{t,\star}
-
V_{\mathrm n}^{t,\star}
\right\|_\infty
}
\nonumber\\
&\changed{
=
\left\|
\mc T_{\mathrm r}^tV_{\mathrm r}^{t,\star}
-
\mc T_{\mathrm n}^tV_{\mathrm n}^{t,\star}
\right\|_\infty
}
\nonumber\\
&\changed{
\le
\left\|
\mc T_{\mathrm r}^tV_{\mathrm r}^{t,\star}
-
\mc T_{\mathrm r}^tV_{\mathrm n}^{t,\star}
\right\|_\infty
+
\left\|
\mc T_{\mathrm r}^tV_{\mathrm n}^{t,\star}
-
\mc T_{\mathrm n}^tV_{\mathrm n}^{t,\star}
\right\|_\infty
}
\nonumber\\
&\changed{
\le
\gamma
\left\|
V_{\mathrm r}^{t,\star}
-
V_{\mathrm n}^{t,\star}
\right\|_\infty
+
\Delta_{r,t}
+
\gamma
\left\|V_{\mathrm n}^{t,\star}\right\|_\infty
\Delta_{P,t}.
}
\end{align}
\changed{Applying \eqref{eq:fixed_point_bounds} and rearranging gives}
\begin{equation}
\changed{
\left\|
V_{\mathrm r}^{t,\star}
-
V_{\mathrm n}^{t,\star}
\right\|_\infty
\le
\frac{
\Delta_{r,t}
+
\gamma M_t\Delta_{P,t}
}{
1-\gamma
}.
}
\end{equation}
\changed{Consequently, if $\Delta_{r,t}$ and $M_t\Delta_{P,t}$ simultaneously converge to zero as $t\to\infty$, then}
\begin{equation}
\changed{
\left\|
V_{\mathrm r}^{t,\star}
-
V_{\mathrm n}^{t,\star}
\right\|_\infty
\to0.
}
\end{equation}
\end{proof}

\begin{proof}[Proof of Corollary~\ref{cor:asym_opt}]
\changed{Let $\Omega_0$ denote the probability one event on which $U_v^t\to\{\theta_v^\star\}$ for every node $v\in\mc V$. Fix an outcome $\omega\in\Omega_0$. Because $\Theta$ is finite and the sequence $\{U_v^t(\omega)\}_{t\ge0}$ is nonempty and nested, convergence to the singleton $\{\theta_v^\star\}$ implies that there exists a finite time $T_v(\omega)$ such that}
\begin{equation}
\changed{
U_v^t(\omega)
=
\{\theta_v^\star\}
\qquad
\text{for all }t\ge T_v(\omega).
}
\end{equation}

\changed{Because the node set $\mc V$ is finite,}
\begin{equation}
\changed{
T(\omega)
=
\max_{v\in\mc V}T_v(\omega)
}
\end{equation}
\changed{is finite. Hence, for every $t\ge T(\omega)$ and every node $v\in\mc V$,}
\begin{equation}
\changed{
U_v^t(\omega)
=
\{\theta_v^\star\},
\qquad
\widehat\theta_{t,v}
=
\theta_v^\star.
}
\end{equation}

\changed{For any state $s\in\mc S$, any action $a\in\mc A(s)$, and any $t\ge T(\omega)$, the inner minimization in the robust Bellman operator is therefore taken over a singleton. Thus,}
\begin{align}
\changed{
(\mc T_{\mathrm r}^tV)(s)
}
&\changed{
=
\max_{a\in\mc A(s)}
\min_{\theta\in\{\theta_{v(s)}^\star\}}
\left\{
r^t(s,a;\theta)
+
\gamma
\sum_{s'\in\mc S}
P(s'\mid s,a;\theta)V(s')
\right\}
}
\nonumber\\
&\changed{
=
\max_{a\in\mc A(s)}
\left\{
r^t(s,a;\theta_{v(s)}^\star)
+
\gamma
\sum_{s'\in\mc S}
P(s'\mid s,a;\theta_{v(s)}^\star)V(s')
\right\}
}
\nonumber\\
&\changed{
=
(\mc T_{\mathrm n}^{t,\theta^\star}V)(s).
}
\end{align}

\changed{Therefore, for every $t\ge T(\omega)$,}
\begin{equation}
\changed{
\mc T_{\mathrm r}^t
=
\mc T_{\mathrm n}^{t,\theta^\star}.
}
\end{equation}
\changed{Since both operators are $\gamma$ contractions and their fixed points are unique,}
\begin{equation}
\changed{
V_{\mathrm r}^{t,\star}
=
V_{\mathrm n}^{t,\theta^\star,\star}
\qquad
\text{for all }t\ge T(\omega).
}
\end{equation}
\changed{It follows immediately that}
\begin{equation}
\changed{
\left\|
V_{\mathrm r}^{t,\star}
-
V_{\mathrm n}^{t,\theta^\star,\star}
\right\|_\infty
\to0
}
\end{equation}
\changed{almost surely.}

\changed{Finally, suppose that the maximizing action of the true nominal operator is unique at every state. Since the robust and nominal operators coincide for all $t\ge T(\omega)$, their maximizing actions also coincide. Under the fixed tie-breaking rule in Assumption~\ref{ass:greedy}, the corresponding robust and nominal greedy policies are therefore identical for every sufficiently large planning time.}
\end{proof}

\begin{proof}[Proof of Proposition~\ref{prop:safety}]
On the event $\mc E_t$, we have $\theta_v^\star\in U_v^t$ for every $v\in\mc V$. \changed{For each feasible pair $(s,a)\in\mc K$, let $v(s)$ denote the node associated with state $s$. By the definition of the local joint model set $\mc M_{s,a}^t$,}
\begin{equation}
\changed{
\left(
r^t(s,a;\theta_{v(s)}^\star),
P(\cdot\mid s,a;\theta_{v(s)}^\star)
\right)
\in
\mc M_{s,a}^t.
}
\end{equation}
\changed{Therefore, on $\mc E_t$, the true collection of local reward and transition pairs, denoted by}
\begin{equation}
\changed{
M_t^\star
=
\left\{
\left(
r^t(s,a;\theta_{v(s)}^\star),
P(\cdot\mid s,a;\theta_{v(s)}^\star)
\right)
\right\}_{(s,a)\in\mc K},
}
\end{equation}
\changed{belongs to the rectangular joint model set}
\begin{equation}
\changed{
\mc M_t
=
\prod_{(s,a)\in\mc K}\mc M_{s,a}^t.
}
\end{equation}

\changed{For any stationary policy $\pi_t$ and state $s\in\mc S$, the true expected discounted value can therefore be written as}
\begin{equation}
\changed{
V_{\pi_t,t}^{\theta^\star}(s)
=
V_{\pi_t,t}^{M_t^\star}(s).
}
\end{equation}
\changed{Because $M_t^\star\in\mc M_t$ on $\mc E_t$, the minimum over $\mc M_t$ cannot exceed the value obtained under the true model. Hence, for every state $s\in\mc S$,}
\begin{equation}
\changed{
V_{\pi_t,t}^{\theta^\star}(s)
\ge
\min_{M\in\mc M_t}V_{\pi_t,t}^{M}(s)
=
\underline V_{\pi_t,t}^{\mathrm r}(s).
}
\end{equation}

\changed{The same event $\mc E_t$ guarantees this inequality simultaneously for all states. Therefore,}
\begin{align}
&\changed{
\Pr\left(
\forall s\in\mc S:
V_{\pi_t,t}^{\theta^\star}(s)
\ge
\underline V_{\pi_t,t}^{\mathrm r}(s)
\right)
}
\nonumber\\
&\changed{
\qquad\ge
\Pr(\mc E_t)
=
1-\delta_t.
}
\end{align}
\changed{This establishes the stated probabilistic robust value guarantee.}
\end{proof}

\begin{proof}[Proof of Proposition~\ref{prop:complexity}]
\changed{Fix a planning time $t$. During the corresponding replan, the credible sets $U_v^t$, reward quantities, transition kernels, and mission statistics are fixed. Let $V_k$ denote the value function after the $k$th robust Bellman iteration. Because $\mc T_{\mathrm r}^t$ is a $\gamma$ contraction,}
\begin{equation}
\changed{
\left\|
V_{k+1}-V_k
\right\|_\infty
\le
\gamma^k
\left\|
V_1-V_0
\right\|_\infty.
}
\end{equation}
\changed{Let $C_t$ be a finite constant satisfying $\|V_1-V_0\|_\infty\le C_t$. To ensure that the value update is no greater than $\varepsilon$, it is sufficient to choose $N_\varepsilon$ such that}
\begin{equation}
\changed{
\gamma^{N_\varepsilon}C_t
\le
\varepsilon.
}
\end{equation}
\changed{Taking logarithms gives}
\begin{equation}
\changed{
N_\varepsilon
\ge
\frac{
\log(C_t/\varepsilon)
}{
-\log\gamma
}.
}
\end{equation}
\changed{Since $-\log\gamma\ge1-\gamma$ for $\gamma\in(0,1)$, the required number of iterations satisfies}
\begin{equation}
\changed{
N_\varepsilon
=
\mathcal O\left(
\frac{\log(C_t/\varepsilon)}{1-\gamma}
\right).
}
\end{equation}

\changed{We next determine the cost of one robust Bellman sweep. For every feasible pair $(s,a)\in\mc K$, the robust backup evaluates}
\begin{equation}
\changed{
r^t(s,a;\theta)
+
\gamma
\sum_{s'\in\mc S}
P(s'\mid s,a;\theta)V_k(s')
}
\end{equation}
\changed{for each prototype $\theta\in U_{v(s)}^t$. Assuming that the prototype-dependent reward quantities are precomputed, the dominant computation is the transition expectation. For a fixed triple $(s,a,\theta)$, this requires}
\begin{equation}
\changed{
\mathcal O(d_{s,a,\theta})
}
\end{equation}
\changed{operations, where}
\begin{equation}
\changed{
d_{s,a,\theta}
=
\left|
\operatorname{supp}
P(\cdot\mid s,a;\theta)
\right|.
}
\end{equation}
\changed{Therefore, the total cost of evaluating all local robust backups during one Bellman sweep is}
\begin{equation}
\changed{
\mathcal O\left(
\sum_{(s,a)\in\mc K}
\sum_{\theta\in U_{v(s)}^t}
d_{s,a,\theta}
\right).
}
\label{eq:proof_single_sweep_cost}
\end{equation}
\changed{The comparisons required for the inner minimizations and outer maximizations contribute only lower order terms because every valid transition distribution has nonempty support.}

\changed{By the definitions of $\overline{|U^t|}$ and $\overline d_t$,}
\begin{align}
&\changed{
|\mc K|\,
\overline{|U^t|}\,
\overline d_t
}
\nonumber\\
&\changed{
=
|\mc K|
\left(
\frac{1}{|\mc K|}
\sum_{(s,a)\in\mc K}
|U_{v(s)}^t|
\right)
\left(
\frac{
\displaystyle
\sum_{(s,a)\in\mc K}
\sum_{\theta\in U_{v(s)}^t}
d_{s,a,\theta}
}{
\displaystyle
\sum_{(s,a)\in\mc K}
|U_{v(s)}^t|
}
\right)
}
\nonumber\\
&\changed{
=
\sum_{(s,a)\in\mc K}
\sum_{\theta\in U_{v(s)}^t}
d_{s,a,\theta}.
}
\end{align}
\changed{Thus, one robust Bellman sweep has computational cost}
\begin{equation}
\changed{
\mathcal O\left(
|\mc K|\,
\overline{|U^t|}\,
\overline d_t
\right).
}
\end{equation}

\changed{Multiplying the cost of one sweep by the required number of value iteration steps gives the total per replan cost}
\begin{equation}
\changed{
\mathcal O\left(
N_\varepsilon
|\mc K|\,
\overline{|U^t|}\,
\overline d_t
\right).
}
\end{equation}
\changed{For dense transition kernels, $d_{s,a,\theta}\le|\mc S|$ for every feasible pair and prototype, and therefore $\overline d_t=\mathcal O(|\mc S|)$. Substitution yields the worst-case bound}
\begin{equation}
\changed{
\mathcal O\left(
N_\varepsilon
|\mc K|\,
\overline{|U^t|}\,
|\mc S|
\right).
}
\end{equation}
\changed{This completes the proof.}
\end{proof}

\section*{Acknowledgments}
This project was supported by NSF IUCRC Phase I: Center for Autonomous Air Mobility and Sensing (CAAMS) Award No. 2137195. The authors also thank Mengmeng Li from École Polytechnique Fédérale de Lausanne (EPFL) for valuable discussions and feedback that helped improve the theoretical components of this work. The authors used ChatGPT (OpenAI) for grammar and language polishing. All content was reviewed and approved by the authors.

\bibliography{sample}
\end{document}